\documentclass[10pt]{amsart}

\usepackage{amssymb, amsmath, amsthm, amsfonts}

\usepackage{mathrsfs}
\usepackage{extarrows}
\usepackage{amsthm}
\usepackage{amssymb}

\usepackage[all]{xy}

\usepackage{verbatim}
\usepackage{hyperref}
\hypersetup{
    colorlinks=true,       
    linkcolor=blue,          
    citecolor=blue,        
    filecolor=blue,      
    urlcolor=blue           
}

\newtheorem{theorem}{Theorem}[section]
\newtheorem{cor}[theorem]{Corollary}
\newtheorem{lem}[theorem]{Lemma}
\newtheorem{prop}[theorem]{Proposition}
\newtheorem{conj}[theorem]{Conjecture}

\theoremstyle{definition}
\newtheorem{defn}[theorem]{Definition}
\newtheorem{ex}[theorem]{Example}

\theoremstyle{remark}
\newtheorem{rem}[theorem]{Remark}
\newtheorem{rems}[theorem]{Remarks}

\newcommand{\EE}{\mathcal{E}}
\newcommand{\OO}{\mathscr{O}}

\newcommand{\HH}{\mathrm{H}}
\newcommand{\End}{\mathrm{End}}
\newcommand{\Hom}{\mathrm{Hom}}

\newcommand{\bA}{\mathbb{A}}

\newcommand{\bP}{\mathbb{P}}

\newcommand{\bN}{\mathbb{N}}
\newcommand{\bZ}{\mathbb{Z}}
\newcommand{\bQ}{\mathbb{Q}}
\newcommand{\bR}{\mathbb{R}}
\newcommand{\bC}{\mathbb{C}}
\newcommand{\bV}{\mathbb{V}}

\newcommand{\cA}{\mathcal{A}}
\newcommand{\cO}{\mathcal{O}}
\newcommand{\cD}{\mathcal{D}}
\newcommand{\cE}{\mathcal{E}}
\newcommand{\cF}{\mathcal{F}}
\newcommand{\cG}{\mathcal{G}}
\newcommand{\cH}{\mathcal{H}}
\newcommand{\cL}{\mathcal{L}}
\newcommand{\cN}{\mathcal{N}}

\newcommand{\cV}{\mathcal{V}}

\newcommand{\arrow}{\rightarrow}
\newcommand{\rk}{\mathrm{rk}\,}
\newcommand{\Sh}{\mathrm{Sh}}
\newcommand{\Gl}{\mathrm{Gl}\,}
\newcommand{\Sl}{\mathrm{Sl}}
\newcommand{\Is}{\mathrm{Is}}
\newcommand{\Id}{\mathrm{Id}}
\newcommand{\Aut}{\mathrm{Aut}}

\newcommand{\Lie}{\mathrm{Lie}}
\newcommand{\Exc}{\mathrm{Exc}}
\newcommand{\ad}{\mathrm{ad}}
\newcommand{\Ad}{\mathrm{Ad}}

\newcommand{\Proj}{\mathrm{Proj}}
\newcommand{\Sym}{\mathrm{Sym}}
\newcommand{\Res}{\mathrm{Res}}
\newcommand{\sbt}{\,\begin{picture}(-1,1)(-1,-3)\circle*{3}\end{picture}\ }

\title[A strong hyperbolicity property]{A strong hyperbolicity property of locally symmetric varieties}
\author{Yohan Brunebarbe}
\date{\today}

\begin{document}

\maketitle

\begin{abstract}
We show that all subvarieties of a quotient of a bounded symmetric domain by a sufficiently small arithmetic discrete group of automorphisms are of general type. This result corresponds through the Green-Griffiths-Lang's conjecture to a well-known result of Nadel.
\end{abstract}

\section{Introduction}

A complex analytic space $X$ is called hyperbolic in the sense of Brody if there is no non-constant holomorphic map $\bC \arrow X$. Given a projective complex variety $X$, one can measure the deviation from Brody-hyperbolicity of the corresponding analytic space by introducing its exceptional subvariety $\Exc(X) \subset X$, which is by definition the Zariski closure of the union of the images of all non-constant holomorphic maps $\bC \arrow X$. A famous conjecture of Green-Griffiths and Lang says that this subset has an interpretation in the realm of algebraic geometry:

\begin{conj}[Green-Griffiths, Lang, cf. \cite{Green-Griffiths79, Lang86}]\label{Lang_geometric}
Let $X$ be an irreducible projective complex variety. Then $X$ is of general type if and only if $\Exc(X) \neq X$.
\end{conj}

Recall that an irreducible smooth projective complex variety $X$ is said of general type if it has enough pluricanonical forms to make the canonical rational maps $X \dasharrow \bP( \HH^0(X, \omega_X ^{\otimes m})^{\vee})$ birational onto their images for $m \gg 1$. In that case, every smooth projective complex variety birational to $X$ is of general type too. An irreducible complex algebraic variety $X$ is said of general type if any smooth projective complex variety birational to $X$ is of general type.\\

Observe that conjecture \ref{Lang_geometric} implies that given a projective complex variety $X$, the irreducible components of its exceptional locus $\Exc(X)$ are not of general type, and that any irreducible subvariety of $X$ not contained in $\Exc(X)$ is of general type.\\

The conjecture \ref{Lang_geometric} is known for some special classes of varieties, including subvarieties of abelian varieties \cite{Bloch26, Ochiai77, Kawamata80}, smooth projective surfaces with a big cotangent bundle \cite{Bogomolov77, McQuillan98} and generic hypersurfaces of high degree \cite{Diverio-Merker-Rousseau10}.\\

In this article, we investigate conjecture \ref{Lang_geometric} for compactifications of arithmetic locally symmetric varieties and their subvarieties. Recall that an arithmetic locally symmetric variety is by definition a complex analytic space which is isomorphic to a quotient of a bounded symmetric domain $\cD$ by a torsion-free arithmetic lattice $\Gamma \subset \Aut(\cD)$, see section \ref{reminder_Arithmetic locally symmetric varieties} for a reminder. By a theorem of Baily-Borel \cite{Baily-Borel}, every arithmetic locally symmetric variety admits a canonical (and in fact unique) structure of smooth quasi-projective variety. The exceptional subvariety of their compactifications has been studied by Nadel (see also \cite{Noguchi91, Hwang-To06, Rousseau13}):

\begin{theorem}[Nadel, \cite{Nadel}] Let $\cD$ be a bounded symmetric domain and $\Gamma \subset \Aut(\cD)$ an arithmetic lattice. Then there exists a subgroup $\Gamma' \subset \Gamma$ of finite index such that the image $f(\bC)$ of any non-constant holomorphic map $f : \bC \arrow \overline{\Gamma' \backslash \cD}$, where $\overline{\Gamma' \backslash \cD}$ is any compactification of $\Gamma' \backslash \cD$, is contained in $\overline{\Gamma' \backslash \cD} - \Gamma' \backslash \cD$.
\end{theorem}

In other words, the exceptional subvariety of any compactification $\overline{\Gamma' \backslash \cD}$ is included in $\overline{\Gamma' \backslash \cD} - \Gamma' \backslash \cD$. Our main result is the corresponding statement predicted by conjecture \ref{Lang_geometric}:

\begin{theorem}[Main result]\label{main_theorem} Let $\cD$ be a bounded symmetric domain and $\Gamma \subset \Aut(\cD)$ an arithmetic lattice. Then there exists a subgroup $\Gamma' \subset \Gamma$ of finite index such that all subvarieties of $\Gamma' \backslash \cD$ are of general type.
\end{theorem}  

\begin{rems}
\begin{enumerate}
\item This is in fact a consequence of a stronger statement, see theorem \ref{precise_main_result}. 
\item That such a $\Gamma' \backslash \cD$ is of general type was already known of Tai \cite{AMRT} and Mumford \cite{Mumford77}, and our proof of theorem \ref{main_theorem} build partly on Mumford's ideas. Also, Brylinski \cite{Brylinski79} has shown that the moduli space of curves with a sufficiently high dihedral level structure is of general type. A stronger result follows from the theorem \ref{main_theorem_for_Ag} below and Torelli's theorem.
\end{enumerate}
\end{rems}

As a direct application of Borel's algebraization theorem \cite[theorem 3.10]{Borel72}, one obtains:

\begin{cor}
Let $\cD$ be a bounded symmetric domain and $\Gamma \subset \Aut(\cD)$ an arithmetic lattice. Then there exists a subgroup $\Gamma' \subset \Gamma$ of finite index such that any (not necessarily compact) complex algebraic variety endowed with a generically immersive holomorphic map to $\Gamma' \backslash \cD$ is of general type.
\end{cor} 

Note that in general it is necessary to take a finite index subgroup as shown by the example $ \Sl(2, \bZ) \backslash \Delta \simeq \bC$. Nonetheless, by \cite[theorem 3.1]{BKT} and its proof, a \textit{compact} complex analytic variety endowed with a generically immersive holomorphic map to a quotient of a bounded symmetric domain by a torsion-free discrete group of automorphisms is of general type.\\

In the special case where $\cD$ is Harish-Chandra's realization of the Siegel half-space of rank $g$, we prove the more precise statement (cf. section \ref{case_of_Ag}):

\begin{theorem}\label{main_theorem_for_Ag} 
For any $g \geq 1$ and any $n > 12 \cdot g$, every subvariety of $\cA_g(n)$ is of general type. 
\end{theorem}  
Here $\cA_g(n)$ denotes the moduli stack of principally polarized abelian varieties of dimension $g$ with a level-$n$ structure, which is a smooth quasi-projective variety for $n \geq 3$. By Torelli's theorem, theorem \ref{main_theorem_for_Ag} implies an analogous statement for the moduli stack of curves with level structures.\\

Our main result has some (partly conjectural) implications for Shimura varieties that we now briefly discuss. Let $(\mathbf{G}, X)$ be a Shimura datum: $\mathbf{G}$ is a reductive $\bQ$-group and $X$ a $\mathbf{G}(\bR)$-conjugation class of morphisms $\mathbf{S} = \Res_{\bC / \bR} \mathbf{G}_m \arrow \mathbf{G}_{\bR}$ satisfying Deligne's axioms, cf. \cite{Deligne71, Deligne77}. If $K$ is a compact-open subgroup of $\mathbf{G}(\bA_f)$, one has the corresponding Shimura variety:
\[ \Sh_K(\mathbf{G}, X)_{\bC} = \mathbf{G}(\bQ) \backslash (X \times \mathbf{G}(\bA_f)/K), \] 
which is a finite disjoint union of arithmetic locally symmetric varieties when $K$ is net. Its Baily-Borel compactification $\overline{\Sh_K(\mathbf{G}, X)_{\bC}}^{BB}$ is a normal projective complex variety which contains $\Sh_K(\mathbf{G}, X)_{\bC}$ as a Zariski-open dense subset. The boundary $\overline{\Sh_K(\mathbf{G}, X)_{\bC}}^{BB} - \Sh_K(\mathbf{G}, X)_{\bC}$ is stratified by Shimura subvarieties, and by an iterative application of Theorem \ref{main_theorem}, one easily gets

\begin{cor}\label{main_theorem_for_BB} 
Let $(\mathbf{G}, X)$ be a Shimura datum and $K$ be a compact-open subgroup of $\mathbf{G}(\bA_f)$. There exists a compact-open subgroup $K' \subset K \subset \mathbf{G}(\bA_f)$ such that any subvarieties of $\overline{\Sh_{K'}(\mathbf{G}, X)_{\bC}}^{BB}$ is of general type.  
\end{cor}

Following Lang, the exceptional subvariety should also have an arithmetic significance:

\begin{conj}[Bombieri, Lang]
Let $X$ be a projective complex variety. For any field $F$ finitely generated over $\bQ$ with an embedding $\sigma : F \arrow \bC$ and any variety $X_F$ defined over $F$ such that $X_F \otimes_{\sigma} \bC = X$, the set of $F$-rational points of $X$ lying outside of $\Exc(X)$ is finite.  
\end{conj}
It is not difficult to verify that if the conjecture \ref{Lang_geometric} is true, then $\Exc(X)$ is defined over any field of definition of $X$. 
By the work of Shimura, Deligne, Milne, Shih and Borovoi, $Sh_K(\mathbf{G}, X)_{\bC}$ admits a canonical model over the reflex field $E(\mathbf{G}, X)$ associated to the Shimura datum $(\mathbf{G}, X)$. Moreover, its Baily-Borel compactification $\overline{\Sh_K(\mathbf{G}, X)}^{BB}$ is defined over the same field. Therefore, as a consequence of Bombieri-Lang conjecture, one obtains the following

\begin{conj}
Let $(\mathbf{G}, X)$ be a Shimura datum and $K$ be a compact-open subgroup of $\mathbf{G}(\bA_f)$. There exists a compact-open subgroup $K' \subset K \subset \mathbf{G}(\bA_f)$ such that for any finitely generated extension $F$ of $E(\mathbf{G}, X)$, the set $\overline{\Sh_{K'}(\mathbf{G}, X)}^{BB}(F)$ of $F$-points of $\overline{\Sh_{K'}(\mathbf{G}, X)}^{BB}$ is finite. 
\end{conj} 
 
See \cite{Ullmo-Yafaev10} for a partial result in direction of this conjecture.\\

In the case of $\cA_g$, in view of our Theorem \ref{main_theorem_for_Ag}, the Bombieri-Lang conjecture would have the following consequence (compare with \cite{AMV16, AV16}):

\begin{conj}
For any $g \geq 1$, any $n > 12 \cdot g$ and any field $F$ finitely generated over $\bQ$, there is only a finite number of isomorphism classes of principally polarized abelian varieties defined over $F$ of dimension $g$ with a level-$n$ structure.   
\end{conj}

\vspace{0.5cm}

Our proof of the theorem \ref{main_theorem} relies strongly on Hodge theory. It uses a semi-negativity result for subsheaves of system of log Hodge bundles contained in the kernel of the Higgs field due to Zuo \cite{Zuoneg} and the author \cite{Brunebarbe}, cf. theorem \ref{nef}. It also uses as a main input the existence on bounded symmetric domains of a special class of systems of Hodge bundles said of Calabi-Yau type, that were first introduced by Gross \cite{Gross94} and further studied by Sheng-Zuo \cite{Sheng-Zuo10} and Friedman-Laza \cite{Friedman-Laza13}. Their construction together with their remarkable properties are recalled in section \ref{Calabi-Yau PVHS}. To help the reader who is not familiar with automorphic objects on arithmetic locally symmetric varieties, we have gathered the necessary background material in section \ref{reminder_Arithmetic locally symmetric varieties}.\\
Our main construction (cf. section \ref{Proof of the main theorem}) is reminiscent of a strategy introduced by Viehweg and Zuo in their study of hyperbolicity properties of moduli spaces of canonically polarized varieties \cite{Viehweg-Zuo02}. Accordingly, it is necessary to consider the cotangent bundle and its symmetric powers, even if one wants ultimately
prove some positivity properties of the canonical bundle. The bridge from the latter to the former is crossed thanks to an important result of Campana-Peternell and Campana-P\u{a}un, cf. theorem \ref{Campana-Paun}.

\subsection{Notations}
In this paper, a smooth log pair $(X, D)$ is a smooth complex algebraic variety $X$ together with $D \subset X$ a union of smooth divisors crossing normally. A log pair $(X, D)$ is said projective when $X$ is projective. A morphism of log pairs $f : (X,D) \arrow (Y,E)$ is a morphism $f : X\arrow Y$ such that $f^{-1}(E)\subset D$. A (projective smooth) log-compactification of a smooth complex variety $U$ is a projective smooth log pair $(X,D)$ with an identification $X - D \simeq U$.
In the sequel, all varieties will supposed to be irreducible.

\section{Different notions of positivity for torsion-free sheaves}

In this section, we recall for the reader convenience different positivity notions for torsion-free sheaves on smooth projective complex varieties that we will use later in this paper.

\subsection{} We begin with some notions due to Viehweg. For details and proofs, the reader is referred to \cite[lemma 1.4]{Viehweg83} and \cite[ p.59-67]{Viehweg_book}.

\begin{defn}
Let $X$ be a complex quasi-projective scheme. A coherent sheaf $\cF$ on $X$ is globally generated at a point $x \in X$ if the natural map $\HH^0(X, \cF) \otimes_{\bC} \cO_X \arrow \cF$ is surjective at $x$.
\end{defn}

\begin{defn}[Viehweg]
Let $X$ be a smooth projective complex variety and $\cF$ a torsion-free sheaf on $X$. Let $i : V  \hookrightarrow X$ denotes the inclusion of the biggest open subset on which $\cF$ is locally free. 
\begin{enumerate}
\item We say that $\cF$ is weakly positive over the dense open subset $U \subset V$ if for every ample invertible sheaf $\cH$ on $X$ and every positive integer $\alpha  > 0$ there exists an integer $\beta > 0$ such that $\widehat{S}^{\alpha \cdot \beta} \cF \otimes_{ \cO_X} {\cH}^{\beta}$ is globally generated over $U$.
\item We say that $\cF$ is Viehweg-big over the dense open subset $U \subset V$ if for any line bundle $\cH$ there exists $\gamma > 0$ such that $\widehat{S}^{\gamma} \cF \otimes \cH^{-1}$ is weakly positive over $U$.
\end{enumerate}
Here the notation $\widehat{S}^{k} \cF $ stands for the reflexive hull of the sheaf $S^{k} \cF $, i.e. $\widehat{S}^{k} \cF = i_* ( S^{k} i^*\cF)$.\\
We say that $\cF$ is weakly positive (resp. Viehweg big) if there exists a dense open subset $U \subset V$ such that $\cF$ is weakly positive (resp. Viehweg big) over $U$. 
\end{defn}

\begin{rem}
If $\cF$ is locally free, then $\cF$ is nef if and only if it is weakly positive over $X$, and $\cF$ is ample if and only if it is Viehweg-big over $X$.
\end{rem}

\begin{lem}[Viehweg]\label{Viehweg} Let $\cF$ and $\cG$ be torsion-free sheaves on a smooth projective complex variety $X$.

\begin{enumerate}

\item If $\cF \arrow \cG$ is a morphism, surjective over $U$, and if $\cF$ is weakly positive over $U$, then $\cG$ is weakly positive over $U$.

\item Let $f : Y \arrow X$ be a morphism between two smooth projective complex varieties. If $\cF$ is weakly positive over $U \subset X$ and $f^{-1}(U)$ is dense in $Y$, then $f^* \cF /(f^* \cF)_{tors}$ is weakly positive over $f^{-1}(U)$.

\item If $\cF \arrow \cG$ is a morphism, surjective over $U$, and if $\cF$ is Viehweg-big over $U$, then $\cG$ is Viehweg-big over $U$.

\item If $\cF$ is weakly positive and $\cH$ is a big line bundle, then $\cF \otimes \cH$ is Viehweg-big.

\item Let $f : Y \arrow X$ be a morphism between two smooth projective complex varieties, which is finite over an open $V \subset X$. If $\cF$ is Viehweg-big over $U$ and $f^{-1}(U \cap V)$ is dense in $Y$, then $f^* \cF /(f^* \cF)_{tors}$ is Viehweg-big over $f^{-1}(U \cap V)$.
\end{enumerate}
\end{lem}

\subsection{} We introduce now a weaker notion of bigness. Let $\cE$ be a vector bundle on a smooth projective complex variety $X$. Let $\pi : \bP(\cE) := \Proj_{\cO_X}(\Sym \, \cE) \arrow X$ be the projective bundle of one-dimensional quotients of $\cE$ and $\cO_{\cE}(1) $ be the tautological line bundle which fits in an exact sequence $ \pi^* \cE \arrow  \cO_{\cE}(1) \arrow 0$.

\begin{lem}\label{lemma_bigness}
The following assertions are equivalent:
\begin{enumerate}
\item The line bundle $ \cO_{\cE}(1)$ is Viehweg-big,
\item  for some (resp. any) big line bundle $\cH$, there exists a injective map $ 0 \arrow \cH \arrow S^k \cE$ for some $k >0$, 
\item for some (resp. any) Viehweg-big torsion-free sheaf $\cF$, there exists a non-zero map $\cF \arrow S^k \cE$ for some $k >0$. 
\end{enumerate}
\end{lem}

\begin{defn}
A vector bundle $\cE$ on a smooth projective complex variety $X$ is called big (in the sense of Hartshorne) if it satisfies the equivalent conditions of the lemma \ref{lemma_bigness}.
\end{defn}
The lemma \ref{lemma_bigness} implies that a Viehweg-big vector bundle is big, but the converse is not true (consider for example the rank $2$ vector bundle $ \cO(1) \oplus  \cO(-1)$ on $\bP^1$). Note however that the two notions coincide for line bundles.
 
\begin{proof}[Proof of lemma \ref{lemma_bigness}]
First observe that given a (non necessarily Viehweg-big) torsion-free sheaf $\cF$ on $X$ and locally-free sheaf $\cE$ such that the line bundle $ \cO_{\cE}(1)$ is Viehweg-big, there exists a non-zero map $\cF \arrow S^k \cE$ for any sufficiently big $k$. Indeed, it is sufficient to show the existence of a section of $ \cO_{\cE}(k) \otimes \pi^* \cF^{\vee}$ for any $k \gg 1$, which follows from \cite[example $2.2.9$]{Lazarsfeld04I}. In particular, this shows that $(1) \implies (3)$.\\
Let us now show $(3) \implies (2)$. Let $i : V  \hookrightarrow X$ be the inclusion of the biggest open subset on which $\cF$ is locally free. Given any ample line bundle $\cH$, there exists $\gamma > 0$ such that $\widehat{S}^{\gamma} \cF \otimes \cH^{-1}$ is weakly positive. Therefore, there exists an integer $\beta > 0$ such that $\widehat{S}^{2 \cdot \beta} (\widehat{S}^{\gamma} \cF \otimes \cH^{-1}) \otimes_{ \cO_X} {\cH}^{\beta} \simeq \widehat{S}^{2 \cdot \beta} (\widehat{S}^{\gamma} \cF ) \otimes_{ \cO_X} {\cH}^{-2 \beta + \beta}  $ is generically globally generated, as well as the quotient sheaf $\widehat{S}^{2 \cdot \beta \cdot \gamma}\cF  \otimes_{ \cO_X} {\cH}^{- \beta }  $. This shows the existence of a generically surjective map $ \oplus \cH^{ \beta } \arrow \widehat{S}^{2 \cdot \beta \cdot \gamma} \cF$. On the other hand, the non-zero map $\cF_{|V} \arrow S^k \cE_{|V}$ corresponds to a non-zero map $(\pi^* \cF)_{|\pi^{-1}(V)} \arrow  \cO_{\cE}(k)_{|\pi^{-1}(V)} $, which in turn provides a non-zero map $(S^{2 \cdot \beta \cdot \gamma} \pi^* \cF )_{|\pi^{-1}(V)} \arrow  \cO_{\cE}(2 \cdot \beta \cdot \gamma \cdot k)_{|\pi^{-1}(V)} $, or equivalently a non-zero map $S^{2 \cdot \beta \cdot \gamma} \cF_{|V} \arrow  S^{2 \cdot \beta \cdot \gamma \cdot k} \cE_{|V}$. Finally, by composing the generically surjective map $ \oplus \cH^{ \beta } \arrow \widehat{S}^{2 \cdot \beta \cdot \gamma} \cF$ with the non-zero map $i_*(S^{2 \cdot \beta \cdot \gamma} \cF_{|V}) \arrow  i_*(S^{2 \cdot \beta \cdot \gamma \cdot k} \cE_{|V})$, we get a non-zero map $ \oplus \cH^{ \beta } \arrow  S^{2 \cdot \beta \cdot \gamma \cdot k} \cE$, hence a non-zero map $ \cH^{ \beta } \arrow  S^{2 \cdot \beta \cdot \gamma \cdot k} \cE$.\\
Let us finally show $(2) \implies (1)$, following \cite[example 6.1.23]{Lazarsfeld04II}. Assume that there exists an injective map $ 0 \arrow \cH \arrow S^k \cE$ for some $k >0$ (using Kodaira's lemma, cf. \cite[proposition 2.2.6]{Lazarsfeld04I}, one can assume that $\cH$ is ample). Equivalently, the line bundle $ \cO_{\cE}(k) \otimes \pi^* \cH^{-1}$ has a non-zero section. On the other hand, as $ \cO_{\cE}(1)$ is relatively ample, $ \cO_{\cE}(1) \otimes \pi^* \cH$ is ample (cf. \cite[proposition 1.7.10]{Lazarsfeld04I}). It follows that  $ \cO_{\cE}(k + 1) = (\cO_{\cE}(k) \otimes \pi^* \cH^{-1}) \otimes \cO_{\cE}(1) \otimes \pi^* \cH$ is big.
\end{proof}

\subsection{Complex algebraic varieties with maximal cotangent dimension}

\begin{defn}
A complex algebraic variety $X$ is said to have maximal cotangent dimension if any smooth projective complex variety birational to $X$ has a big cotangent bundle.
\end{defn}

If $X$ is a smooth projective complex variety with maximal cotangent dimension, then any smooth projective complex variety birational to $X$ has the same property, cf. \cite{Sakai}.

\begin{lem}\label{generically finite_mcd}
Let $f : X \arrow Y$ be a generically finite and dominant algebraic map between two complex algebraic varieties. If $Y$ has maximal cotangent dimension, then the same is true for $X$.
\end{lem}

\begin{theorem}[Campana-Peternell, Campana-Paun, \cite{Campana-Paun1}, see also {\cite[corollary 2.24]{Claudon-Bourbaki}} and the references therein] \label{Campana-Paun}
Any complex algebraic variety of maximal cotangent dimension is of general type. 
\end{theorem}

\section{Variations of Hodge structures}

The axioms of variations of Hodge structures were introduced by Griffiths in the $60$'s. In this paper, we will need a slight generalization which does not ask for the existence of a real structure on the underlying local system, as in \cite{Deligne87, Simpson92}. 
 
\subsection{Definitions}\label{Def_VHS}
\begin{defn}
Let $V$ be a complex vector space of finite dimension. A (complex) Hodge structure (of weight zero) on $V$ is a decomposition $ V = \bigoplus_{p \in \bZ} V^p $.
A polarization of a complex Hodge structure is a non-degenerate hermitian form $h$ on $V$ making the decomposition $ V = \bigoplus_{p \in \bZ} V^p $ orthogonal,
such that the restriction of $h$ to $V^p$ is positive definite for $p$ even and negative definite for $p$ odd. 
The associated Hodge filtration is the decreasing filtration $F$ on $V$ defined by 
$ F^p := \bigoplus_{ q \geq p} V^q$.
\end{defn}
If we are given a Hodge structure on $V$, then a non-degenerate hermitian form $h$ on $V$ is a polarization exactly when the hermitian form $h_H$ defined by
\begin{equation}\label{Hodge metric}
h_H(u,v):=h(C . u, v)
\end{equation}
for $u,v \in V$ is positive definite. Here $C$ denotes the Weil operator, i.e. the endomorphism of $V$ whose restriction to $V^p$ is the multiplication by $(-1)^p$.\\

\begin{defn}
Let $S$ be a complex manifold. A variation of polarized complex Hodge structures ($\bC$-PVHS) on $S$ is the data of a holomorphic vector bundle $\cV$ equipped with an integrable connection $\nabla$, a $\nabla$-flat non-degenerate hermitian form $h$ and for all $s \in S$ a decomposition of the fibre $\cV_s = \bigoplus_{p \in \bZ} \cV^p$ satisfying the following axioms:

\begin{enumerate}
\item for all $s \in S$, the decomposition $\cV_s = \bigoplus_{p \in \bZ} \cV_s^p $ defines a Hodge structure polarized by $h_s$,
\item  the Hodge filtration $\cF$ varies holomorphically with $s$, 
\item (Griffiths' transversality) $\nabla( \cF^p) \subset \cF^{p-1} \otimes \Omega_S^1$ for all $p$.
\end{enumerate}  
\end{defn}

\begin{defn}
Let $S$ be a complex manifold. A Higgs bundle on $S$ is the data of a holomorphic vector bundle $\EE$ equipped with a $1$-form $\theta \in \Omega^1_S \otimes \End(\EE)$, called the Higgs field, which satisfies $\theta \wedge \theta = 0 \in \Omega^2_S \otimes \End(\EE)$.
A system of Hodge bundles on $S$ is a Higgs bundle $(\EE, \theta)$ with a decomposition $\EE = \bigoplus_{p \in \bZ}  \EE^p$ as a sum of holomorphic subvector bundles such that $\theta(\EE^p) \subset \EE^{p-1} \otimes \Omega_S^1$.
\end{defn}

The Higgs field $\theta$ corresponds to an $\cO_S$-linear morphism $\phi : T_S \arrow \End(\EE)$. The condition $\theta \wedge \theta = 0$ implies that for every $k \geq 1$, the induced morphism $\phi_k : T_S^{\otimes k} \arrow \End(\EE)$ factorizes through $S^k T_S$.\\

If $\bV = (\cV, \nabla, \cF, h)$ is a $\bC$-PVHS, we define the corresponding system of Hodge bundles $ (\EE = \bigoplus_{p \in \bZ}  \EE^p , \theta = \oplus_{p \in \bZ} \theta_p)$ by setting:
\begin{equation}\label{system of Hodge bundles}
 \EE^p = \cF^p / \cF^{p+1} \text{\, and \,} \theta_p := Gr_p^{\cF} \theta.  
\end{equation}

\subsection{Canonical extensions}

\begin{defn}[Log $\bC$-PVHS] Let $(X, D)$ be a smooth log pair, and set $U := X -D$. A log complex polarized variation of Hodge structure (log $\bC$-PVHS) on $(X,D)$ consists in the following data:
\begin{itemize}
\item A holomorphic vector bundle $\cV$ on $X$ endowed with a connection $\nabla$ with logarithmic singularities along $D$.
\item An exhaustive decreasing filtration $\cF$ on $V$ by holomorphic subbundles (the Hodge filtration), satisfying Griffiths transversality
\[ \nabla \cF^p \subset \cF^{p-1} \otimes \Omega^1_X(\log D). \]
\item a $\nabla$-flat non-degenerate hermitian form $h$ on $\cV_{|U}$
\end{itemize}
such that $(\cV_{|U}, \nabla, \cF^{\sbt}_{|U}, h)$ is a $\bC$-PVHS on $U$.
\end{defn}

\begin{defn}[Log Higgs bundles and systems of log Hodge bundles] 
A log Higgs bundle on a smooth log pair $(X,D)$ is the data of a holomorphic vector bundle $\EE$ equipped with a $1$-form $\theta \in \Omega^1_X(\log D) \otimes \End(\EE)$ which satisfies $\theta \wedge \theta = 0 \in \Omega^2_X(\log D) \otimes \End(\EE)$.
A system of log Hodge bundles on $(X,D)$ is a log Higgs bundle $(\EE, \theta)$ with a decomposition $\EE = \bigoplus_{p \in \bZ}  \EE^p$ as a sum of holomorphic subvector bundles such that $\theta(\EE^p) \subset \EE^{p-1} \otimes \Omega_X^1(\log D)$.
\end{defn}

As before, one associates to any log $\bC$-PVHS a system of log Hodge bundles by the formula (\ref{system of Hodge bundles}).\\

Let $(X,D)$ be a smooth log pair with $U := X - D$ and $j : U \hookrightarrow X$ the inclusion. Let $(\cV, \nabla ,  \cF^{\sbt}, h)$ be a $\bC$-PVHS on $U$ whose local monodromies around the irreducible components of $D$ are unipotent. Let $\cV^{\geq 0}$ be Deligne's extension of $\cV$ to $X$, i.e. the unique extension for which $\nabla$ has logarithmic poles and the residues are nilpotent, cf. \cite{Deligne70}. We can extend the Hodge filtration by setting:
\begin{equation*}
\cF^p \cV^{\geq 0}  := \cV^{\geq 0} \cap j_* \cF^p,
\end{equation*}
and it turns out that this is still a filtration by vector subbundles (this is a consequence of Schmid's nilpotent orbit theorem \cite{Schmid73}). It follows that $(\cV^{\geq 0}, \nabla, \cF^{\sbt} \cV^{\geq 0} , h)$ is a log $\bC$-PVHS on $(X,D)$. By taking the associated graded object with respect to $\cF^{\sbt}$, we get the so-called canonical extension $(\EE^{\geq 0}, \theta^{\geq 0})$ of $(\EE, \theta)$ to $(X,D)$.

The following result will be one of the crucial point of the proof of theorem \ref{main_theorem}:

\begin{theorem}[Zuo {\cite[theorem 1.2]{Zuoneg}}, see also {\cite[theorem 1.6]{Brunebarbe}}]\label{nef} 
Let $\mathbb{V}$ be a log $\bC$-PVHS on a projective smooth log pair $(X,D)$ with nilpotent residues along the irreducible components of $D$, and denote by $(\EE, \theta)$ the corresponding system of log Hodge bundles.
If $\cA$ is a holomorphic subbundle of $\EE$ contained in the kernel of the Higgs field $\theta$, then its dual $\cA^\vee$ is nef.
\end{theorem}

On easily infers: 

\begin{cor}[cf. {\cite[corollary 3.4]{Popa-Wu16}}]\label{weakly negativity of kernels}
Let $\mathbb{V}$ be a log $\bC$-PVHS on a projective smooth log pair $(X,D)$ with unipotent local monodromies around the irreducible components of $D$, and denote by $(\EE, \theta)$ the corresponding system of log Hodge bundles. If $\cA$ is a coherent subsheaf of $\EE$ contained in the kernel of $\theta$, then its dual ${\cA}^{\vee}$ is a weakly positive torsion-free sheaf.
\end{cor}

\section{Arithmetic locally symmetric varieties and automorphic objects}\label{reminder_Arithmetic locally symmetric varieties}
We collect in this section some facts about arithmetic locally symmetric varieties that will be used in the sequel. References include \cite[chapter III $\S 2$]{AMRT} and \cite{Deligne71, Deligne77, Helgason78, Milne13, Mok_book}.

\subsection{Generalities on bounded symmetric domains}
\begin{defn}
A bounded symmetric domain $\cD$ is a bounded connected open subset in some $\bC^n$ such that every $z \in D$ is an isolated fixed point of an holomorphic involution (equivalently, there exits an holomorphic involution fixing $z$ and acting by $- \mathrm{Id}$ on the tangent space at $z$).
\end{defn}

Let $\Aut(\cD)$ be the group of holomorphic automorphisms of $\cD$, and $\Is(\cD)$ be the group of isometries of $\cD$ with respect to the riemannian structure defined by the Bergman metric (cf. \cite[chapter 4]{Mok_book}). Both $\Aut(\cD)$ and $\Is(\cD)$ are semi-simple real Lie groups with finitely many connected components, and $\Aut(\cD)^+ = \Is(\cD)^+$ is a semi-simple real Lie group with trivial center. In particular, $\Is(\cD)^+ = \Ad(\mathfrak{g}) \subset \Aut(\mathfrak{g})$ where $\mathfrak{g}$ is the Lie algebra of $\Is(\cD)$ and $^+$ denotes the connected component of the identity in the euclidean topology. In fact, it is known that $\Is(\cD) = \Aut(\mathfrak{g})$. Let $\mathbf{G}$ be the connected component of the identity subgroup of the real algebraic group $\mathbf{Aut}(\mathfrak{g})$, so that $ \mathbf{G}(\bR)^+ = \Aut(\cD)^+ \subset \mathbf{Aut}(\mathfrak{g})(\bR)$. We call $\mathbf{G}$ the real algebraic group associated to $\cD$. It is a connected semi-simple real algebraic group of adjoint type. For every $z \in \cD$, the subgroup $K_z \subset \mathbf{G}(\bR)^+$ of biholomorphisms fixing $z$ is a maximal compact subgroup of $\mathbf{G}(\bR)^+$. Choosing a basepoint $z$ for $\cD$, we have an identification $\cD = \mathbf{G}(\bR)^+/ K_z$.\\ 

Let $\cD$ be a bounded symmetric domain and $\mathbf{G}$ be the corresponding real algebraic group of adjoint type. We call $\cD$ irreducible when $\mathbf{G}$ is simple. In general, $\cD$ can be decomposed uniquely as a product of irreducible bounded symmetric domains, and this decomposition corresponds to the decomposition of $\mathbf{G}$ as a direct product of its simple subgroups.
The rank of $\cD$ is by definition the rank of the real algebraic group $\mathbf{G}$, i.e. the dimension of a maximal split torus. It is denoted by $\rk \cD$.

\begin{defn}
A subgroup $\Gamma$ of $\Aut(\cD)$ is called arithmetic if there exists a $\bQ$-form $\mathbf{G}_{\bQ} $ of $\mathbf{G}$ and an embedding $\mathbf{G}_{\bQ} \hookrightarrow \mathbf{Gl}_n$ defined over $\bQ$ such that $\Gamma$ is commensurable with $\mathbf{G}_{\bQ}(\bQ) \cap \mathbf{Gl}_n(\bZ)$, i.e. the intersection is of finite index in each. This property turns out to be independent of the embedding. 
\end{defn}

Given a $\bQ$-group $\mathbf{G}_{\bQ} $, a subgroup $\Gamma \subset \mathbf{G}_{\bQ}(\bQ)$ is called neat (cf. \cite[\S 17.1]{Borel_book69}) if for any representation $\mathbf{G}_{\bQ} \arrow \mathbf{Gl}_n$ defined over $\bQ$ and any element $\gamma \in \Gamma$, the subgroup of $\bC^*$ generated by the eigenvalues of the automorphism of $\bC^n$ associated to $\gamma$ is torsion-free.
In particular, $\Gamma$ is torsion-free. Moreover, any arithmetic group admits a finite-index neat subgroup.

\begin{defn}
An arithmetic locally symmetric variety is a complex analytic space which is isomorphic to a quotient of a bounded symmetric domain $\cD$ by
 a torsion-free arithmetic lattice $\Gamma \subset \Aut(\cD)$.
\end{defn}

It follows from the work of Baily-Borel \cite{Baily-Borel} that every arithmetic locally symmetric variety $X = \Gamma  \backslash \cD$ admits a canonical compactification by a normal projective variety that we denote $\overline{X}^{BB}$. In particular, $X$ admits a canonical (and in fact unique) structure of smooth quasi-projective variety. However, Igusa and others have shown that the singularities of $\overline{X}^{BB}$ are very complicated in general. In order to address this problem, Mumford et al. \cite{AMRT} have introduced the collection of so-called toroidal compactifications of $X$, which are algebraic spaces with only quotient singularities. These compactifications are not unique, they depend on the choice of a combinatorial data $\Sigma$. Moreover, when $\Gamma$ is neat, it is always possible to find a toroidal compactification $\overline{X}^{\Sigma}$ of $X$ such that $(\overline{X}^{\Sigma}, D)$ is a projective log-compactification of $X$, where $D := \overline{X}^{\Sigma}-X$. For any toroidal compactification $\overline{X}^{\Sigma}$, the identity map $X \arrow X$ can always be extended to a holomorphic map $\overline{X}^{\Sigma} \arrow\overline{X}^{BB}$.   \\

We denote by $\mathbf{U}(1)$ the real algebraic group $\mathbf{SO}(2)$ and by $U(1)= \mathbf{U}(1)(\bR)$ the corresponding real Lie group. For every $z \in \cD$, there exists a unique group homomorphism $u_z : U(1) \arrow K_z \subset \mathbf{G}(\bR)^+$ such that for any $\lambda \in U(1)$, the automorphism of the tangent space $T_z \cD$ at $z$ induced by $u_z(\lambda)$ is the multiplication by $\lambda$ coming from the complex structure on $\cD$. The subgroup $K_z$ is the centralizer of $u_z(U(1))$ in $ \mathbf{G}(\bR)^+$ (in particular $K_z$ is connected) and, if $\cD$ is irreducible, $u_z(U(1))$ is the center of $K_z$. All the morphisms $u_z$ for $z \in \cD$ are conjugated under $\mathbf{G}(\bR)^+$, hence $\cD$ can be identified with the $\mathbf{G}(\bR)^+$-conjugacy class of one such $u_z : U(1) \arrow  \mathbf{G}(\bR)^+$. Observe that every $u_z$ comes from a real algebraic morphism $u_z : \mathbf{U}(1) \arrow \mathbf{G}$ which satisfies the following properties:
\begin{enumerate}
\item the induced action of $ \mathbf{U}(1)$ on the complex Lie algebra $\Lie(\mathbf{G}(\bC))$ has weights $-1, 0, 1$,
\item $\ad(u_z(-1))$ is a Cartan involution of $\mathbf{G}$,
\item $u_z(-1)$ does not project to $1$ in any simple factor of $\mathbf{G}$. 
\end{enumerate}

Recall that a Cartan involution of a (non necessarily connected) real linear algebraic group $\mathbf{G}$ is an involution $\sigma$ of $\mathbf{G}$ such that the real form $\mathbf{G}^{\sigma}$ of $\mathbf{G}_{\bC}$ associated to the complex conjugation $ g \mapsto \sigma(\bar{g})$ is compact, i.e. $\mathbf{G}^{\sigma}(\bR)$ is compact and meets all connected components of $\mathbf{G}^{\sigma}_{\bC}(\bC) = \mathbf{G}_{\bC}(\bC)$. Note that property $(3)$ is equivalent to ask that $\mathbf{G}$ has no compact factor.\\

Conversely, if $\mathbf{G}$ is a connected real adjoint algebraic group and $u:  \mathbf{U}(1) \arrow \mathbf{G}$ is a morphism which satisfies the preceding conditions $(1)$, $(2)$ and $(3)$, then the $\mathbf{G}(\bR)^+$-conjugacy class $\cD$ of $u$ has a natural structure of a bounded symmetric domain for which $u(\lambda)$ acts on $T_{u} \cD$ as multiplication by $\lambda$. Moreover, $ \mathbf{G}(\bR)^+ = \Aut(\cD)^+$.

\begin{defn} A (pointed) \textit{real} Hodge datum of hermitian type is a couple $( \mathbf{G} , u)$ composed of a connected \textit{real} adjoint algebraic group $\mathbf{G}$ and a non-trivial morphism $u:  \mathbf{U}(1) \arrow \mathbf{G}$ which satisfies the preceding conditions $(1)$, $(2)$ and $(3)$. 
A (pointed) \textit{rational} Hodge datum of hermitian type is a couple $( \mathbf{G} , u)$ composed by a connected \textit{rational} adjoint algebraic group $\mathbf{G}$ and a non-trivial morphism $u:  \mathbf{U}(1) \arrow \mathbf{G}_{\bR}$ which satisfies the preceding conditions $(1)$, $(2)$ and $(3)$. 
\end{defn}

Our preceding discussion shows that there is a natural correspondence between isomorphism classes of pointed bounded symmetric domains and isomorphism classes of real Hodge data of hermitian type. \\

Let $( \mathbf{G} , u)$ be a real Hodge datum of hermitian type, and denote by $K$ the centralizer of $u(U(1))$ in $G:= \mathbf{G}(\bR)^+$. The element $u(-1) \in  G$ defines an involution $\sigma$ of $\mathfrak{g} = \Lie(G)$. The subspace where $\sigma$ acts by the identity is $\Lie(K)$, and we denote by $\mathfrak{p}$ the subvector space where $\sigma$ acts by $- \Id$, so that $\Lie(G) = \Lie(K) \oplus \mathfrak{p}$.
Let also $ \mathfrak{p}_{\bC} =  \mathfrak{p}_+ \oplus  \mathfrak{p}_- $ be the decomposition into $\pm i$-eigenspaces for $\sigma_{\bC}$, and $P_{\pm}$ be the subgroups of $ G$ generated by $\exp( \mathfrak{p}_{\pm})$. Then $K_{\bC}$ normalizes $P_{\pm}$ and $Q := K_{\bC} \cdot P_+$ is a parabolic subgroup of $ G$. Here, $K_{\bC}$ is the complexification of $K$.  This provides a $G$-equivariant open immersion $\cD = G/K \hookrightarrow \hat{\cD} = G_{\bC} / Q$ of $\cD$ in its compact dual $\hat{\cD}$.

\subsection{Automorphic bundles}
We follow the presentation of \cite{Zucker81}.

Let $(\mathbf{G}, u)$ be a real Hodge datum of hermitian type, and $\cD$ be the associated bounded symmetric domain. Let $G$ be a connected real semi-simple Lie group with finite center such that $G^{ad} = \mathbf{G}(\bR)^+ $. Let $K$ be a maximal compact subgroup of $G$, so that there is an identification $\cD = G/K $ and $K$ is the stabilizer in $G$ of a point $z_0 \in \cD$. Finally, let $\Gamma$ be a torsion-free discrete subgroup of $G$.

\subsubsection{} Let $\cE$ be a $G$-equivariant $C^{\infty}$ real (resp. complex) vector bundle on $\cD$, i.e. a $C^{\infty}$-vector bundle on $\cD$ endowed with a left action of $G$ which covers the $G$-action on $\cD$. Then the subgroup $K \subset G$ acts on the fibre of $\cE$ at $z_0$. Conversely, given any finite-dimensional real or complex representation $\tau : K \arrow \Gl(V)$, we define a $G$-equivariant $C^{\infty}$-vector bundle $\cE(\tau)$ on $\cD$ whose fibre at $z_0$ is identified with $V$ by taking the quotient of $G \times V$ by the following right-action of $K$ (which covers the natural right-action of $K$ on $G$):
\[   (g,v)\cdot k :=  (g k, \tau(k^{-1} ) v). \]

\begin{ex}
The left-action of $K$ on $\mathfrak{p}_{\bC}$ coming from the adjoint representation yields the complexified tangent bundle of $\cD$.
\end{ex}

For any torsion-free discrete subgroup $\Gamma$ of $G$, the $C^{\infty}$-vector bundle $\cE(\tau)$ descends to a $C^{\infty}$-vector bundle on $\Gamma \backslash \cD$ that we denote abusively $\cE(\tau)$.

\subsubsection{}  For any torsion-free discrete subgroup $\Gamma$ of $G$ and any finite-dimensional real or complex representation $\psi : \Gamma \arrow \Gl(V)$, we get a flat vector bundle $\Psi(\psi)$ on $\Gamma \backslash \cD$ defined as the quotient of $\cD \times V$ under the left-action of $\Gamma$:
\[  \gamma \cdot (x , v) = (\gamma \cdot x, \psi(\gamma) v). \]

\begin{prop} Let $\rho : G \arrow  \Gl(V)$ be a finite-dimensional real or complex representation (which induces by restriction representations of $\Gamma$ and $K$). Then the mapping $\Xi : G \times V \arrow G \times V$ defined by $\Xi (g , v) =( g , \rho(g)^{-1} v)$ induces an isomorphism of $C^{\infty}$-vector bundles $\Xi : \Psi( \rho_{| \Gamma}) \cong  \cE(\rho_{|K} )$.
\end{prop}

\subsubsection{}
Let $G_{\bC}$ be the complexification of $G$. As before, the $G_{\bC}$-equivariant holomorphic vector bundles on $\hat{\cD}$ are in correspondence with finite-dimensional complex representations of $Q$. Given a representation $\sigma : Q \arrow \Gl(V)$, we denote by $\cE(\sigma)$ the associated holomorphic vector bundle on $\hat{\cD}$ with a left action of $G_{\bC}$ which covered the $G_{\bC}$-action on $\hat{\cD}$. It induces by restriction a holomorphic vector bundle on $\cD$ with a left action of $G$ which covered the $G$-action on $\cD$, which descends to a holomorphic vector bundle on $\Gamma \backslash \cD$.

\begin{prop} Let $\rho : G \arrow  V$ be a finite-dimensional complex representation which then induces representations of $\Gamma$ and $Q$. Then the mapping $\Xi : G_{\bC} \times V \arrow G_{\bC} \times V$ defined by $\Xi (g , v) = (g , \rho(g)^{-1} v)$ induces an isomorphism of holomorphic vector bundles $\Xi : \Psi( \rho_{| \Gamma}) \cong  \cE({\rho_{\bC}}_{|Q} )$.
\end{prop}

\subsubsection{}
Every complex representation $\tau : K \arrow \Gl(V)$ extends canonically to a morphism $K_{\bC} \arrow \Gl(V)$. Using the morphism $Q \arrow K_{\bC} \arrow 1$, we obtain a representation $\sigma :  Q \arrow \Gl(V)$, which by restriction to $K \subset Q$ gives back the representation $\tau$. This procedure endows canonically $\cE(\tau)$ with a holomorphic structure. Note however that when $\tau : K \arrow \Gl(V)$ is the restriction to $K$ of a morphism $\rho : G \arrow \Gl(V)$, this holomorphic structure on $\cE(\tau)$ differs in general from the holomorphic structure on $\cE(\rho_{|K} )$.

\subsubsection{Canonical extensions of automorphic vector bundles}
Let $\tau : K \arrow \Gl(V)$ be a complex representation, and let $\cE(\tau)$ be the corresponding $G$-equivariant holomorphic vector bundle on $\cD$. Any choice of a $K$-invariant positive definite hermitian form on $V$ induces a hermitian metric $h$ on $\cE(\tau)$. Assume moreover that $(\mathbf{G}, u)$ comes from a rational Hodge datum of hermitian type $(\mathbf{G}_{\bQ}, u)$. For any torsion-free arithmetic subgroup $\Gamma \subset \mathbf{G}_{\bQ}(\bQ)$, the hermitian holomorphic vector bundle $(\cE(\tau),h)$ descends to the arithmetic locally symmetric variety $X = \Gamma \backslash \cD$. If $\overline{X}$ is a toroidal compactification of $X$, Mumford has shown that $\cE(\tau)$ admits a unique extension on $\overline{X}$ such that $h$ is a singular metric good on $\overline{X}$ (see \cite[section 1]{Mumford77} for the meaning of "good" and \cite[theorem 3.1]{Mumford77}). We call this extension Mumford's canonical extension of $\cE(\tau)$ to $\overline{X}$.

\subsection{The canonical automorphic line bundle}\label{The canonical automorphic line bundle}
Let $( \mathbf{G} , u)$ be a rational Hodge datum of hermitian type, and $\cD$ be the associated bounded symmetric domain. Choosing a basepoint $z$ of $\cD$, denote by $K \subset  \mathbf{G}(\bR)^+$ the centralizer of $u_z(U(1))$ in $ \mathbf{G}(\bR)^+$, so that $\cD =  \mathbf{G}(\bR)^+ / K$. When $\cD$ is moreover irreducible, the group of isomorphism classes of $1$-dimensional complex representations of $K$ is infinite cyclic, hence so is the group of isomorphism classes of $G$-equivariant line bundles on $\cD$. We denote by $\cL$  the generator such that the canonical bundle of $\cD$ is a positive power of $\cL$. When $\cD$ is reducible, one defines $\cL$ to be the tensor product of the $G$-equivariant line bundles obtained as before. For any torsion-free arithmetic subgroup $\Gamma \subset \mathbf{G}(\bQ)$, the holomorphic line bundle $\cL$ descends to the arithmetic locally symmetric variety $\Gamma \backslash \cD$, and has a canonical extension to any toroidal compactification of $\Gamma \backslash \cD$ that we still denote by $\cL$. 

\begin{theorem}\label{amplitude_L} Let $\overline{X}$ be a toroidal compactification of an arithmetic locally symmetric variety $X = \Gamma \backslash \cD$, and let $\cL$ be Mumford's canonical extension to $\overline{X}$ of the canonical automorphic line bundle.
Then $\cL$ is ample with respect to $X$.
\end{theorem}
\begin{proof}
This is well-known and can be proved as follows in the irreducible case. Let $D := \overline{X} - X$. First observe that a positive tensor power of $\cL_{|X}$ is isomorphic to the canonical bundle $\omega_X$ of $X$ and that the canonical extension of $\omega_X$ to $\overline{X}$ is $\omega_X(D)$, hence it is sufficient to show that $\omega_X(D)$ is ample with respect to $X$. On the other hand, if we denote by $\overline{X}^{BB}$ the Baily-Borel compactification of $X$, then $\omega_X(D)$ is the pull-back of the canonical bundle $\omega_{\overline{X}^{BB}}$ of $\overline{X}^{BB}$ by the canonical map $\overline{X} \arrow \overline{X}^{BB}$, cf. \cite[proposition 3.4]{Mumford77}. But $\omega_{\overline{X}^{BB}}$  is known to be ample.
\end{proof}

\subsection{Automorphic $\bC$-PVHS}\label{automorphic C-PVHS}

Let $\mathbf{G}$ be a semi-simple real algebraic group whose adjoint group $\mathbf{G}^{ad}$ admits a real Hodge datum of hermitian type $(\mathbf{G}^{ad}, u)$. Let $\cD$ be the corresponding bounded symmetric domain. In this section, we describe briefly a procedure which associates to an irreducible finite-dimensional complex representation $\rho : \mathbf{G}(\bC) \arrow \Gl(V)$ a $\bC$-PVHS on $\cD$. For a complete treatment, we refer the reader to \cite[section 4]{Zucker81}.\\

The following elementary lemma rephrases the definition of a complex polarized Hodge structure in terms of representation theory.
\begin{lem} Let $V$ be a complex vector space of finite dimension endowed with a non-degenerate hermitian form $h$.
The data of a Hodge structure on $V$ polarized by $h$ is equivalent to the data of a morphism $u : \mathbf{U}(1) \arrow \mathbf{U}(V,h)$ such that $\ad (u(-1))$ is a Cartan involution of $\mathbf{U}(V,h)$. 
\end{lem} 

The morphism $u : \mathbf{U}(1) \arrow \mathbf{G}$ permits to define a polarized complex Hodge structure on $V$ as follows. First, we obtain a decomposition $V = \bigoplus_{n \in \bZ} V(n)  $ by setting
\[   V(n) := \{ v \in V \, | \, \forall \lambda \in  \mathbf{U}(1)(\bC)=\bC^*, (\rho \circ u (\lambda))(v) = \lambda^n \cdot v \}.   \]

\begin{ex}\label{definition_of-mu}
When $\rho : \mathbf{G}(\bC) \arrow \Gl(\mathfrak{g})$ is the adjoint representation, we get three summands: $\mathfrak{p}^+ = V< \mu>, \mathfrak{k}_{|\bC} = V<0>, \mathfrak{p}^- = V< -\mu>$, where $\mu$ is a positive integer which is equal to $1$ when $\mathbf{G}$ is of adjoint type.
\end{ex}

Let $\mathbf{G}^c$ be the compact form of $\mathbf{G}_{\bC}$ corresponding to $C :=\ad (u(-1))$. Let $h_H$ be any $\mathbf{G}^c$-invariant hermitian positive definite form on $V$ (unique up to multiplication by non-zero constant) and set $h(x,y) = h_H(C . x, y)$ for every $x,y \in V$. It is a $\mathbf{G}$-invariant hermitian form on $V$. By construction, the morphism $u$ takes its values in $\mathbf{U}(V,h)$ and $\ad (u(-1))$ is a Cartan involution of $\mathbf{U}(V,h)$, hence it defines a Hodge structure on $V$ polarized by $h$. Moreover, as $h(x, g C g^{-1}y) = h(g^{-1}x,  C g^{-1}y)$ for any $g \in \mathbf{G}(\bR)$, $h$ is a polarization of any Hodge structure of $V$ associated to a morphism $\mathbf{U}(1) \arrow \mathbf{G}$ which is $\mathbf{G}(\bR)$-conjugated to $u$.\\

Any $z \in \cD$ corresponds to a morphism $u_z: \mathbf{U}(1) \arrow \mathbf{G}^{ad}$, which can be uniquely lifted to a morphism $u_z': \mathbf{U}(1) \arrow \mathbf{G}$. Moreover, $\ad (u_z(-1))$ is a Cartan involution of $\mathbf{G}^{ad}$ if and only if $\ad (u_z'(-1))$ is a Cartan involution of $\mathbf{G}$. By the preceding discussion, we get for any $z \in \cD$ a complex Hodge structure on $V$ polarized by $h$. When $ \mathbf{G} =  \mathbf{G}^{ad}$, or more generally when $\mu =1$ (cf. example \ref{definition_of-mu} for the definition of $\mu$), we define in this way a $\bC$-PVHS on the constant local system with fibre $V$ on $\cD$. However, if $\mu >1$, Griffiths transversality is not satisfied in general. This problem is addressed in \cite{Zucker81} by introducing the following ad hoc numbering (see also \cite[definition 4.2.4]{Klingler11} for a more canonical solution).
The set of integers $n \in \bZ$ such that $V(n) \neq 0$ is of the form $\{ \lambda, \lambda - \mu,  \lambda - 2 \mu, \cdots ,  \lambda - m\mu   \}$ for some integers $\lambda, m \geq 0$. 
Then define the Hodge decomposition $V = \bigoplus_{p \in \bN} V^p  $ by $V^p = V( \lambda - p\mu )$.
We obtain in this way a $\bC$-PVHS on $\cD$, endowed with a left action of $G$ which covers the $G$-action on $\cD$, hence descends to any quotient $\Gamma \backslash \cD$ by a torsion-free discrete subgroup $\Gamma$ of $G$. This $\bC$-PVHS is called the automorphic $\bC$-PVHS (or the locally homogeneous $\bC$-PVHS) associated to the representation $\rho : \mathbf{G}(\bC) \arrow \Gl(V)$. It is real whenever $\rho$ comes from a representation $\mathbf{G} \arrow \mathbf{Gl}(V_{\bR})$, where $V_{\bR}$ is a real form of $V$: $V = V_{\bR} \otimes \bC$.\\

Assume now that $(\mathbf{G}^{ad}, u)$ comes from a rational Hodge datum of hermitian type $(\mathbf{G}^{ad}_{\bQ}, u)$, and that $\Gamma \subset \mathbf{G}^{ad}_{\bQ}(\bQ)$ is an arithmetic subgroup. Let $\overline{X}$ be a smooth toroidal compactification of $X := \Gamma \backslash \cD$, and set $D := \overline{X} - X$. The automorphic $\bC$-PVHS associated to the representation $\rho : \mathbf{G}(\bC) \arrow \Gl(V)$ has unipotent monodromy around $D$. Let $\bV$ be its canonical extension to $\overline{X}$ as a log $\bC$-PVHS, and $(\EE, \theta)$ be the corresponding system of log Hodge bundles. The holomorphic vector bundle $\EE^p$ is canonically isomorphic to Mumford's canonical extension of $\EE_{|X}^p$ for every $p$. Moreover, as $K$ is the centralizer of $u(U(1))$ in $ \mathbf{G}(\bR)^+$, every $V^p$ is naturally endowed with an action of $K$, and the corresponding automorphic bundle is canonically isomorphic to $\EE^p_{|X}$. 

\section{Complex PVHS of Calabi-Yau type on bounded symmetric domains}\label{Calabi-Yau PVHS}

\begin{defn} Let $V$ be a finite dimensional complex vector space. A complex Hodge structure $V = \bigoplus_p V^p$ is said of Calabi-Yau type if it is effective (i.e. $V^p = 0$ for $p <0$) and $\dim V^n =1$, where $n$ is the biggest integer $p$ such that $V^p \neq 0$. We call $n$ the weight of the Hodge structure.
\end{defn}

The terminology is of course motivated by the complex Hodge structure on the degree $n$ cohomology group of a compact Calabi-Yau manifold of complex dimension $n$.
We define analogously $\bC$-PVHS of Calabi-Yau type.\\

Let $\cD$ be an irreducible bounded symmetric domain. Let $\mathbf{G}$ be the corresponding adjoint real algebraic linear group, $\mathbf{G}_{\bC}$ be its complexification and $\mathfrak{g}_{\bC}$ be the complex Lie algebra of $\mathbf{G}_{\bC}(\bC) = \mathbf{G}(\bC)$. Let $\mathfrak{h} \subset \mathfrak{b} \subset \mathfrak{g}_{\bC}$ be a Cartan subalgebra contained in a Borel subalgebra, and let $\Delta = \{\alpha_1, \cdots, \alpha_r \}$ be the corresponding set of simple roots of $\mathfrak{h}$ and $\Delta^{\vee} = \{\alpha_1^{\vee}, \cdots, \alpha_r^{\vee} \}$ the corresponding set of simple coroots. Writing the highest root as a linear combination of the simple roots $\alpha = \sum_{1 \leq i \leq r} n_i \alpha_i$, we say that a simple root $\alpha_i$ is special if $n_i =1$. As explained in \cite[♦section 1.2]{Deligne77}, the irreducible bounded symmetric domains such that the corresponding adjoint real algebraic group has complexification $\mathbf{G}_{\bC}$ are canonically in bijection with the special roots in $\Delta$.
If $\alpha_i $ is the special root associated to $\cD$, then we define the corresponding cominuscule weight $\omega_i$ by the conditions $\omega_i(\alpha^{\vee}_j) = \delta_{ij}$. The corresponding fundamental representation of the algebraic universal cover of $\mathbf{G}_{\bC}$ with highest weight $\omega_i$ is called the cominuscule representation. It determines an irreducible $\mathbf{G}(\bR)^+$-equivariant $\bC$-PVHS on $\cD$ by the procedure described in section \ref{automorphic C-PVHS}, which has been studied by Gross \cite{Gross94} when $\cD$ is a tube domain and by Sheng-Zuo \cite{Sheng-Zuo10} in the general case. 

\begin{theorem}[Gross \cite{Gross94}, Sheng-Zuo \cite{Sheng-Zuo10}]
Let $\cD $ be an irreducible bounded symmetric domain. The cominuscule representation determines an irreducible $\mathbf{G}(\bR)^+$-equivariant $\bC$-PVHS of Calabi-Yau type on $\cD$ of weight $n = \rk \cD$ that we denote $\bV_{\cD}$. 
\end{theorem}

\begin{theorem}[Friedman-Laza \cite{Friedman-Laza13}]\label{classification}
Let $\cD $ be an irreducible bounded symmetric domain. For any $n \geq 1$, $S^n \bV_{\cD}$ contains a unique irreducible $\mathbf{G}(\bR)^+$-equivariant $\bC$-PVHS of Calabi-Yau type of weight $  n \cdot \rk \cD$.\\
Conversely, any irreducible $\mathbf{G}(\bR)^+$-equivariant $\bC$-PVHS of Calabi-Yau type on $\cD$ is obtained by this procedure.
\end{theorem}

Therefore, for every $n \geq 1$, there is a unique irreducible $\mathbf{G}(\bR)^+$-equivariant $\bC$-PVHS of Calabi-Yau type of weight $n \cdot \rk \cD$ on $\cD$. The one of minimal weight is $\bV_{\cD}$.

\begin{theorem}[Gross \cite{Gross94}]
Let $\cD $ be an irreducible bounded symmetric domain. If $(\EE, \theta)$ denotes the $\mathbf{G}(\bR)^+$-equivariant system of Hodge bundles which corresponds to the unique irreducible $\mathbf{G}(\bR)^+$-equivariant $\bC$-PVHS of Calabi-Yau type of weight $ w = \rk \cD$, then $\EE^w \simeq \cL$ as $\mathbf{G}(\bR)^+$-equivariant line bundles. Consequently, for any $n \geq 1$, if $(\EE, \theta)$ denotes the $\mathbf{G}(\bR)^+$-equivariant system of Hodge bundles which corresponds to the unique irreducible $\mathbf{G}(\bR)^+$-equivariant $\bC$-PVHS of Calabi-Yau type of weight $ w = n \cdot \rk \cD$, then $\EE^w \simeq \cL^{\otimes \rk \cD }$ as $\mathbf{G}(\bR)^+$-equivariant line bundles.
\end{theorem}

\begin{cor}
Let $\cD $ be a (non necessarily irreducible) bounded symmetric domain. For every $n \geq 1$, there exists an irreducible automorphic $\bC$-PVHS of Calabi-Yau type of weight $ w = n \cdot \rk \cD$. If $(\EE, \theta)$ denotes the corresponding automorphic system of Hodge bundles, then $\EE^w \simeq L^{\otimes \rk \cD }$ as automorphic line bundles.
\end{cor}

\begin{prop}[Gross {\cite[proposition 5.2]{Gross94}}]\label{BTT}
Let $\cD $ be an irreducible bounded symmetric domain. If $(\EE, \theta)$ denotes the $\mathbf{G}(\bR)^+$-equivariant system of Hodge bundles which corresponds to the unique irreducible $\mathbf{G}(\bR)^+$-equivariant $\bC$-PVHS of Calabi-Yau type of weight $ w = \rk \cD$, then the induced $\phi_1 : T_{\cD} \arrow \Hom(\EE^n, \EE^{n-1})$ is an isomorphism of $\mathbf{G}(\bR)^+$-equivariant vector bundles.
\end{prop}

\section{Differential forms with logarithmic poles}\label{Differential forms with logarithmic poles}

\subsection{The logarithmic cotangent bundle}

\begin{defn}
If $(X,D)$ is a smooth log pair, its logarithmic cotangent sheaf $ \Omega^1_{X}(\log D)$ is the $\cO_X$-module whose sections on an open subset $U \subset X$ are the holomorphic $1$-forms $\omega$ on $U - D$ such that $\omega$ and $d \omega$ have at most
a simple pole along $D \cap U$ (cf. \cite[II.3]{Deligne70}; see also \cite[chapter 11]{Iitaka_book}).
\end{defn}

If $X$ is a smooth algebraic variety of dimension $n$, then a coordinate neighborhood of a point $x \in X$ is an affine open neighborhood $U$ of $x$ in $X$ together with functions $ z_1, \cdots, z_n  \in \cO_X(U)$ vanishing at $x$ such that ${\Omega_X^1}_{|U} = \bigoplus_{1 \leq i \leq n} \cO_U \cdot dz_i$. The $n$-tuple $( z_1, \cdots, z_n )$ is called a local coordinate system at $x$.

If $(X,D)$ is a smooth log pair and $U$ is a coordinate neighborhood of $x \in X$ with local coordinate system $( z_1, \cdots, z_n )$ such that $D = \cup_{1 \leq i \leq r}\{ z_i =0 \}$ around $x$, then $ \Omega^1_{X}(\log D)_{|U}$ is the $\cO_U$-submodule of $ \Omega^1_{U} \otimes_{\cO_U} \cO_U (D)$ generated by the $\frac{dz_i}{z_i}$ ($1 \leq i \leq r$) and the $dz_j$ ($r < j \leq \dim X$). In particular, it follows that the $\cO_X$-module $ \Omega^1_{X}(\log D)$ is locally-free.

Note that any morphism of smooth log pairs $f : (Y,E) \arrow (X,D)$ induces a pull-back morphism of $\cO_X$-modules $f^* : \Omega^1_{X}(\log D) \arrow f_* (  \Omega^1_{Y}(\log E)) $.

\begin{lem}\label{log-cotangent_modification} Let $\mu : (X',D') \arrow (X,D)$ be a morphism of smooth log pairs such that the underlying map $\mu : X' \arrow X$ is proper and birational. Then the pull-back morphism
$\mu^* : \Omega^1_{X}(\log D) \arrow \mu_* (  \Omega^1_{X'}(\log D')) $ is an isomorphism of $\cO_X$-modules.
\end{lem}
\begin{proof}
The injectivity of $\mu^*$ is clear. Let us show the surjectivity. Denote by $Z$ the smallest closed subset of $X$ such that $\mu$ induces an isomorphism $X' - \mu^{-1} (Z) \xrightarrow{\sim} X - Z$. Let $U \subset X$ be an open subset and $s$ be a section of $ \Omega^1_{X'}(\log D')$ on $\mu^{-1} (U)$. Restricting $s$ to $ \mu^{-1}(U) - \mu^{-1} (Z)$, we get by the isomorphism above a section $t$ of $ \Omega^1_{X}(\log D)$ defined on $U - Z$. By Hartogs' phenomenon, the section $t$ extends uniquely to a section on $U$ whose pull-back is the section $s$ we began with (as they are generically equal).   
\end{proof}

\begin{lem}\label{log-cotangent_subvariety}
For any morphism of smooth log pairs $f:(Y,E) \arrow (X,D)$ such that the underlying morphism $f$ is proper and maps birationally $Y$ onto its image, the pull-back morphism of $\cO_X$-modules $\Omega^1_{X}(\log D) \xrightarrow{f^*} f_*  \Omega^1_{Y}(\log E)$
is surjective.
\end{lem}
 
 \begin{proof}
First consider the case where the underlying map $f : Y \arrow X$ is a closed immersion of a smooth subvariety (and so $E = D \cap Y$). Then the claim is that for every $x \in X$ there exists a neighbourhood $U$ of $x$ in $X$ such that every section of $\Omega^1_{Y}(\log E)$ on $U \cap Y$ can be lifted to a section of $\Omega^1_{X}(\log D)$ on $U$. This follows easily from the local presentation of the log-cotangent bundle, cf. above.

For the general case, by taking an embedded resolution of singularities of the image of $Y$, one shows the existence of a commutative diagram
\begin{align*}
\xymatrix{
(Y', E')  \ar[r]^{f'} \ar[d]^{\nu} &(X',D')  \ar[d]^{\mu}   \\
(Y,E)   \ar[r]^{f}  & (X,D) }
\end{align*}
where $(Y', E')$ and $(X',D')$ are smooth log pairs, $\mu$ and $\nu$ are proper and birational morphisms of log pairs and $f' : (Y', E') \arrow (X',D')$ is a morphism of log pairs whose image $Y'' := f'(Y')$ is a smooth subvariety of $X'$ such that $Y'' \cap D'$ is a simple normal crossing divisor.
The claim then follows from the special case above and the same kind of arguments used during the proof of lemma \ref{log-cotangent_modification}:

Given $x \in X$, let us show that $f^*$ is surjective at $x$. The only non-trivial case is when $x$ is of the form $x = f(y)$. Let $s$ be a section of $\Omega^1_{Y}(\log E)$ defined in a neighborhood of $y$. Choosing a point $y' \in Y'$ such that $\nu(y') = y$, one get by pulling-back a section $\nu^*s$ of $\Omega^1_{Y'}(\log E')$ defined in a neighborhood of $y'$. Arguing as in the proof of lemma \ref{log-cotangent_modification}, the section $\nu^*s$ can be obtained by pull-back from a section $s''$ of $\Omega^1_{Y''}(\log Y'' \cap D')$ defined in a neighborhood of $f'(y')$ in $Y''$. Then, by applying the special case above, we see that this last section $s''$ is in turn the pull-back of a section $t'$ of $\Omega^1_{X'}(\log D')$ defined in a neighborhood of $f'(y')$ in $X'$. Finally, this section $t'$ is the pull-back of a section $t$ of $\Omega^1_{X}(\log D)$ defined in a neighborhood of $\mu (f'(y')) = f( \nu (y')) = x$ in $X$. We claim that $s = f^*t $ in a neighborhood of $y$. As $\nu^*$ is injective, it is sufficient to show that $\nu^* s = \nu^* ( f^* t) = f'^*( \mu^* t)$ in a neighborhood of $y'$, but this is true by construction. This finishes the proof of the surjectivity.
\end{proof}

\subsection{Highly ramified towers}
This section is strongly inspired by \cite[pp.~269-272]{Mumford77}. Note however that our definition differs slightly from that of \textit{loc.~cit}.

\begin{defn} Let $(\overline{X},D)$ be a projective smooth log pair together with a collection of connected finite \'etale Galois coverings $\pi_{\alpha} : X_{\alpha} \arrow X$, where $X := \overline{X}-D$. Assume that any two covers are dominated by a third one. We say that the tower $\{ \pi_{\alpha} \}$ is highly ramified over $D$ if for any $N \geq 1$, there exists $\alpha$ and $(\overline{X}_{\alpha}, D_{\alpha})$ a projective smooth log-compactification of $X_{\alpha}$ with a morphism of log pairs $(\overline{X}_{\alpha}, D_{\alpha}) \arrow (\overline{X},D)$ which extends $\pi_{\alpha}$, such that the index of ramification of every component of $D_{\alpha}$ is at least $N$. 
\end{defn}

\begin{ex}\label{arithmetic_quotients_are_UR} 
Let $(\mathbf{G}, u)$ be a rational Hodge datum of hermitian type with $\cD$ the associated bounded symmetric domain, and $\Gamma \subset \mathbf{G}(\bQ)$ a torsion-free arithmetic subgroup. By definition, there is an embedding $\sigma : \mathbf{G}_{\bQ} \hookrightarrow \mathbf{Gl}_n$ defined over $\bQ$ such that $\Gamma$ is commensurable with $\mathbf{G}_{\sigma}(\bZ) := \mathbf{G}_{\bQ}(\bQ) \cap \mathbf{Gl}_n(\bZ)$. For every positive integer $n$, let $\Gamma(n)_{\sigma}$ be the normal subgroup of elements of $\Gamma$ which belongs to the kernel of $\mathbf{Gl}_n(\bZ) \arrow \mathbf{Gl}_n(\bZ/ n \bZ)$. This corresponds to a collection of connected finite \'etale Galois coverings $\pi_{n} : X_{n} \arrow X$, where $X_{n}  :=  \Gamma(n)_{\sigma} \backslash \cD$ and $X := X_1$. In this situation, Mumford \cite[pp.~271-272]{Mumford77} has shown that for any (non necessarily smooth projective) toroidal compactification $\overline{X}$ of $X$, the tower $\{ \pi_{n} \}$ is highly ramified over $D := \overline{X} - X$. More precisely, he has shown the stronger property that for any positive integer $N$, one can make the ramification index of all components of the divisor at infinity relatively to the map $ \pi_{n} $ divisible by $N$ by taking a sufficiently big $n$.
\end{ex}

\begin{theorem}\label{Mumford's trick} Let $(\overline{X},D)$ be a projective smooth log pair together with a collection of connected finite \'etale Galois coverings $\pi_{\alpha} : X_{\alpha} \arrow X$, where $X := \overline{X}-D$. Consider the following properties.

\begin{enumerate}
\item The tower $\{ \pi_{\alpha} \}$ is highly ramified over $D$.

\item For every $N \geq 1$, there exists $\alpha$ and $(\overline{X}_{\alpha}, D_{\alpha})$ a projective smooth log-compactification of $X_{\alpha}$ with a morphism of log pairs $(\overline{X}_{\alpha}, D_{\alpha}) \arrow (\overline{X},D)$ which extends $\pi_{\alpha}$ such that for every integers $l,m$ satisfying $\frac{l}{m} \leq N$, the following diagram commutes 
\begin{align*}
\xymatrix{
(\Omega^1_{\overline{X}}(\log D))^{\otimes l}  \otimes_{\OO_{\overline{X}}} \OO_{\overline{X}}(- m D)  \ar[r] \ar[d]^{\pi_{\alpha}^*} & (\Omega^1_{\overline{X}}(\log D))^{\otimes l}   \ar[d]^{\pi_{\alpha}^*}   \\
(\Omega^1_{\overline{X}_{\alpha}} )^{\otimes l}  \ar[r]  & (\Omega^1_{\overline{X}_{\alpha}}(\log D_{\alpha}))^{\otimes l}  }
\end{align*}

\item For every $N \geq 1$, there exists $\alpha$ such that for any projective smooth log pair $(\overline{Y},E)$ equipped with a morphism of log pairs $f:(\overline{Y},E) \arrow (\overline{X},D)$ mapping birationnally $\overline{Y}$ onto its image, for any projective smooth log-compactification $(\overline{Y}', E')$ of an irreducible component of $(\overline{Y}-E) \times_X X_{\alpha}$ and every integers $l$ and $m$ satisfying $\frac{l}{m} \leq N$, the following diagram commutes 
\begin{align*}
\xymatrix{
( \Omega^1_{\overline{Y}}(\log E))^{\otimes l}  \otimes_{\OO_{\overline{Y}}} f^* \OO_{\overline{X}}(- m D)  \ar[r] \ar[d]^{\pi^*} &(\Omega^1_{\overline{Y}}(\log E))^{\otimes l}   \ar[d]^{\pi^*}   \\
( \Omega^1_{\overline{Y}'})^{\otimes l}   \ar[r]  & ( \Omega^1_{\overline{Y}'}(\log E'))^{\otimes l}  }
\end{align*}
\end{enumerate} 
Then $(1) \implies (2) \implies (3)$.
\end{theorem}  

\begin{proof}
We first show the implication $(1) \implies (2) $. By assumption, there exists $\alpha$ and $(\overline{X}_{\alpha}, D_{\alpha})$ a projective smooth log-compactification of $X_{\alpha}$ with a morphism of log pairs $(\overline{X}_{\alpha}, D_{\alpha}) \arrow (\overline{X},D)$ which extends $\pi_{\alpha}$, such that the index of ramification of every component of $D_{\alpha}$ is at least $N$. It follows that 
\[ \pi_{\alpha}^* ((\Omega^1_{\overline{X}}(\log D))^{\otimes l}  \otimes_{\OO_{\overline{X}}} \OO_{\overline{X}}(- m D)) \subset  (\Omega^1_{\overline{X}_{\alpha}}(\log D_{\alpha}))^{\otimes l} \otimes_{\OO_{\overline{X}_{\alpha}}} \OO_{\overline{X}_{\alpha}}(- mN \cdot D_{\alpha})),\] and the latter is included in $(\Omega^1_{\overline{X}_{\alpha}} )^{\otimes l} $ as soon as $m  N \geq l$.\\

Let us now show the implication $(2) \implies (3)$.
Fix $N \geq 1$ and take $\alpha$ as in $(2)$. First observe that the commutation can be verified locally. But it follows from the projection formula and the lemma \ref{log-cotangent_subvariety} that the morphism of $\cO_{\overline{X}}$-modules
\[ \Omega^1_{\overline{X}}(\log D) \otimes_{\OO_{\overline{X}}} \OO_{\overline{X}}(- m D) \xrightarrow{f^*} f_* ( \Omega^1_{\overline{Y}}(\log E)\otimes_{\OO_{\overline{Y}}} f^* \OO_{\overline{X}}(- m D))  \]
is surjective, hence the commutation of $(3)$ is a consequence of the commutation of $(2)$.
\end{proof}

\section{Proof of the main theorem}\label{Proof of the main theorem}

In this section, we give a proof of the following

\begin{theorem}[]\label{precise_main_result}
Any arithmetic locally symmetric variety $X$ admits a finite \'etale cover $X'$ corresponding to a congruence subgroup of $\pi_1(X)$ such that all subvarieties of $X'$ have maximal cotangent dimension. 
\end{theorem}  

This result implies our main theorem \ref{main_theorem} in view of the theorem \ref{Campana-Paun}. \\

Let $(\mathbf{G}, u)$ be the rational Hodge datum of hermitian type, with $\cD$ the associated bounded symmetric domain, and $\Gamma \subset \mathbf{G}(\bQ)$ the torsion-free arithmetic subgroup, such that $X = \Gamma \backslash \cD$. The group $\Gamma$ can be assumed to be neat from the beginning. Let $\overline{X}$ be a smooth projective toroidal compactification of $X$ and set $D := \overline{X} - X$. We denote by $r_{\cD}$ the rank of $\cD$, and by $\cL$ the canonical extension to $\overline{X}$ of the canonical automorphic line bundle on $X$, cf. section \ref{The canonical automorphic line bundle}.\\

A key step in the proof of theorem \ref{precise_main_result} is the following

\begin{prop}\label{main_construction} Let $l$ and $m$ be two integers such that the line bundle $ \cL^{\otimes l} (- m D)$ is big over $X$. Then, for any map of projective smooth log pairs $f : (\overline{Y},E) \arrow (\overline{X},D)$ which is generically finite, the vector bundle $ S^{l \cdot r_{\cD}} \Omega_{\overline{Y}}^1 (\log E) \otimes_{\cO_{\overline{Y}}} f^* \cO_{\overline{X}}(- m D)$ contains a Viehweg-big subsheaf (in particular it is big).
\end{prop}

\begin{proof}
By assumption, the line bundle $\cH := \cL^{\otimes l}(- m D)$ is big over $X$.\\

By Theorem \ref{classification}, there exists an irreducible system of log Hodge bundles $(\EE, \theta)$ on the log pair $(\overline{X},D)$ which is of Calabi-Yau type of weight $w := r_{\cD} \cdot l$. It verifies $\EE^w = \cL ^{\otimes l}$.\\

Let $(\overline{Y}, E)$ be a projective smooth log pair with a generically finite morphism $ f : (\overline{Y},E) \arrow (\overline{X},D)$.
By pulling-back we get a system of logarithmic Hodge bundles of Calabi-Yau type of weight $w$ on the log pair $(\overline{Y},E)$ that we denote $(\EE_{|\overline{Y}}, \theta_{|\overline{Y}})$.\\

For all $k \geq 0$, we have morphisms of $\cO_{\overline{X}}$-modules (cf. section \ref{Def_VHS})
\[ \phi_k :  S^k T_{\overline{Y}}( -\log E) \otimes_{\cO_{\overline{Y}}} \EE_{|\overline{Y}}^w \arrow \EE_{|\overline{Y}}^{w-k}. \] 

Let $k_{\overline{Y}} \leq w$ be the largest $k$ for which the map $ \phi_k $ is non-zero. Note that $k_{\overline{Y}} \geq 1 $ by proposition \ref{BTT}. Denote by $\cN_{\overline{Y}}$ the image of  $\phi_{k_{\overline{Y}}}$. It is a coherent subsheaf of $\EE_{|\overline{Y}}^{w-k_{\overline{Y}}}$ contained in the kernel of $\theta_{|\overline{Y}}$. It follows from corollary \ref{weakly negativity of kernels} that its dual $\cN_{\overline{Y}}^{\vee}$ is a weakly positive torsion-free coherent sheaf. From $\phi_{k_{\overline{Y}}}$ we get a non-zero morphism
\begin{align*}
 \cN_{\overline{Y}}^{\vee} \otimes_{\cO_{\overline{Y}}} \EE_{|\overline{Y}}^{w} \arrow S^{k_{\overline{Y}}} \Omega_{\overline{Y}}^1 (\log E).
\end{align*}
By tensoring with the line bundle $ f^* \cO_{\overline{X}}(- m D)$, this provides a non-zero morphism
\begin{align}\label{Viehweg-Zuo morphism bis}
 \cN_{\overline{Y}}^{\vee} \otimes_{\cO_{\overline{Y}}} f^* \cH \arrow S^{k_{\overline{Y}}} \Omega_{\overline{Y}}^1 (\log E) \otimes_{\cO_{\overline{Y}}}  f^* \cO_{\overline{X}}(- m D).
\end{align}

Observe that $f^* \cH$ is big because $\cH$ is big over $X$, cf. lemma \ref{Viehweg}. It follows by lemma \ref{Viehweg} that the left-hand side of (\ref{Viehweg-Zuo morphism bis}) is Viehweg-big, hence its image in $S^{k_{\overline{Y}}} \Omega_{\overline{Y}}^1 (\log E) \otimes_{\cO_{\overline{Y}}}  f^* \cO_{\overline{X}}(- m D)$ is Viehweg-big (cf. lemma \ref{Viehweg}).

A fortiori, the vector bundle $S^{w} \Omega_{\overline{Y}}^1 (\log E) \otimes_{\cO_{\overline{Y}}} f^* \cO_{\overline{X}}(- m D)$ is big (cf. lemma \ref{lemma_bigness}). 
\end{proof}

Let us now finish the proof of theorem \ref{precise_main_result}. We keep the notations introduced above.\\

Recall that $\cL$ is big over $X$ (cf. theorem \ref{amplitude_L}), hence the same is true for the line bundle $ \cL^{\otimes l} (-  D)$ for $l$ big enough.

Fixing an embedding $\sigma : \mathbf{G}_{\bQ} \hookrightarrow \mathbf{Gl}_n$ defined over $\bQ$, let $\{ \pi_n \}$ be the tower of connected finite \'etale Galois coverings defined in example \ref{arithmetic_quotients_are_UR}. Recall that in this setting the pair $(\overline{X},D)$ is highly ramified over $D$. By Theorem \ref{Mumford's trick}, there exists $n \geq 1$ such that for any projective smooth log pair $(\overline{Y},E)$ equipped with a morphism of log pairs $f:(\overline{Y},E) \arrow (\overline{X},D)$ mapping birationnally $\overline{Y}$ onto its image, for any projective smooth log-compactification $(\overline{Y}', E')$ of an irreducible component of $(\overline{Y}-E) \times_X X_{n}$, the following diagram commutes 

\begin{align*}
\xymatrix{
S^{l \cdot r_{\cD}}  \Omega^1_{\overline{Y}}(\log E)  \otimes_{\OO_{\overline{Y}}} f^* \OO_{\overline{X}}(-  D)  \ar[r] \ar[d]^{\pi_n^*} & S^{l \cdot r_{\cD}}  \Omega^1_{\overline{Y}}(\log E)   \ar[d]^{\pi_n^*}   \\
S^{l \cdot r_{\cD}}  \Omega^1_{\overline{Y}'}   \ar[r]  & S^{l \cdot r_{\cD}}  \Omega^1_{\overline{Y}'}(\log E')  }
\end{align*}

Let $Z$ be a subvariety of $X_n$ of positive dimension. Its image by the proper map $\pi_n$ is a subvariety of $X$. Moreover, the canonical map $Z \arrow \pi_n(Z) \times_X X_n$ identifies $Z$ with an irreducible component of $ \pi_n(Z) \times_X X_n$. 
We can thus apply the Theorem \ref{Mumford's trick} to $(\overline{Y},E)$ a projective smooth log-compactification of a proper resolution of $\pi_n(Z)$, and $(\overline{Y}', E')$ a projective smooth log-compactification of the corresponding irreducible component of $(\overline{Y}-E) \times_X X_{n}$, such that the restriction of $\pi_n$ extends to a map of log pairs $(\overline{Y}', E') \arrow (\overline{Y},E)$. It follows from the proposition \ref{main_construction} that the vector bundle $ S^{l \cdot r_{\cD}} \Omega_{\overline{Y}}^1 (\log E) \otimes_{\cO_{\overline{Y}}} f^* \cO_{\overline{X}}(- D)$ contains a Viehweg-big subsheaf. But the map $\overline{Y}' \arrow \overline{Y}$ is generically finite and surjective, so the diagram above shows that the vector bundle $S^{l \cdot r_{\cD}}  \Omega^1_{\overline{Y}'} $ contains a Viehweg-big subsheaf too (cf. lemma \ref{Viehweg}). This finally shows that $ \Omega^1_{\overline{Y}'} $ is big by the lemma \ref{lemma_bigness}, hence that $Z$ has maximal cotangent dimension.

\section{The case of the Siegel half-spaces}\label{case_of_Ag}

In order to obtain optimized results, we will deal in this section with orbifolds, noticing that all we have done so far make sense in this bigger category.\\

Let $g$ and $n$ be two positive integers. Denote by $\cA_g(n)$ the moduli stack of principally polarized abelian varieties with a level-n structure. Recall that a level-n structure on a principally polarized abelian variety $A$ of dimension$g$ over a field $k$ of characteristic zero is a $2g$-tuple of points in $A(k)$ which generate the subgroup of $n$-torsion points in $A(\bar{k})$ and form a symplectic basis with respect to the Weil pairing.

 Fix a smooth toroidal compactification $\overline{\cA_g}$ of $\cA_g$ and set $D := \overline{\cA_g} - \cA_g$. We denote by $\cL$ the canonical extension to $\overline{\cA_g}$ of the automorphic line bundle. 
If $\bV$ is the canonical weight one $\bQ$-PVHS on $\cA_g$, then the irreducible automorphic $\bC$-PVHS of Calabi-Yau type of weight $g$ is a direct factor of $\Lambda^g \bV$. Denoting by $(\EE,\theta)$ the corresponding system of Hodge bundles, one has $\EE^g = \det(\EE^{1,0}) = \cL_{|\cA_g}$. By taking the canonical extensions of both sides, one gets $\EE^g = \cL$. Recall also  that $\cL^{\otimes (g + 1)}$ is isomorphic to the log-canonical bundle $\omega_{\overline{\cA_g}}(D)$ of $(\overline{\cA_g}, D)$.

Let $\bV$ be the canonical weight one $\bQ$-PVHS on $\cA_g$ and $(\EE_0 = \EE_0^{1,0} \oplus \EE_0^{0,1},\theta)$ be the corresponding system of Hodge bundles. The irreducible automorphic $\bC$-PVHS of Calabi-Yau type of weight $g$ is a direct factor of $\Lambda^g \bV$. Denoting by $(\EE,\theta)$ the corresponding system of Hodge bundles, one has $\EE^g = \det(\EE_0^{1,0}) = \cL_{|\cA_g}$. By taking the canonical extensions of both sides, one gets $\EE^g = \cL$. Recall also  that $\cL^{\otimes (g + 1)}$ is isomorphic to the log-canonical bundle $\omega_{\overline{\cA_g}}(D)$ of $(\overline{\cA_g}, D)$.

\begin{lem}[Nadel {\cite[proof of theorem 3.1]{Nadel}}, see also {\cite[proposition 4.1]{AV16}}] Let $\overline{\cA_g(n)}$ be a smooth toroidal compactification of $\cA_g(n)$ such that the map $\pi_n$ extends to a map $\overline{\pi_n} : \overline{\cA_g(n)} \arrow \overline{\cA_g}$.
Then $\overline{\pi_n}$ is ramified to order at least $n$ over the divisor at infinity $D$.
\end{lem}

\begin{lem}[Weissauer \cite{Weissauer86}]
For any $x \in \cA_g$ and any $\epsilon >0$, there exists $l,m  \in \bN$ with $ (12 + \epsilon) \cdot m \geq l$ and a section of $\cL^{\otimes l}(- m D)$ non-vanishing at $x$. 
\end{lem}

As $\cL$ is big over $\cA_g$, it follows that the line bundle $\cL^{\otimes l}(- m D)$ is big over $\cA_g$ for any $l,m  \in \bN$ with $ 12 \cdot m <  l$. Arguing as in section \ref{Proof of the main theorem}, we get the
\begin{theorem}
Let $g \geq 1$. For every $n >  12 \cdot g$, all subvarieties of $\cA_g(n)$ are of maximal cotangent dimension (hence a fortiori of general type).
\end{theorem}

\bibliographystyle{alpha}
\bibliography{biblio}

\begin{thebibliography}{AMPVA16}

\bibitem[AMPVA16]{AMV16}
Dan Abramovich, Keerthi Madapusi~Pera, and Anthony V{\'a}rilly-Alvarado.
\newblock Level structures on abelian varieties and {V}ojta's conjecture.
\newblock {\em arXiv:1604.04571}, April 2016.

\bibitem[AMRT75]{AMRT}
A.~Ash, D.~Mumford, M.~Rapoport, and Y.~Tai.
\newblock {\em Smooth compactification of locally symmetric varieties}.
\newblock Math. Sci. Press, Brookline, Mass., 1975.
\newblock Lie Groups: History, Frontiers and Applications, Vol. IV.

\bibitem[AVA16]{AV16}
Dan Abramovich and Anthony V{\'a}rilly-Alvarado.
\newblock Level structures on abelian varieties, {K}odaira dimensions, and
  {L}ang's conjecture.
\newblock {\em arXiv: 1601.02483}, January 2016.

\bibitem[BB66]{Baily-Borel}
W.~L. Baily, Jr. and A.~Borel.
\newblock Compactification of arithmetic quotients of bounded symmetric
  domains.
\newblock {\em Ann. of Math. (2)}, 84:442--528, 1966.

\bibitem[BKT13]{BKT}
Yohan Brunebarbe, Bruno Klingler, and Burt Totaro.
\newblock Symmetric differentials and the fundamental group.
\newblock {\em Duke Math. J.}, 162(14):2797--2813, 2013.

\bibitem[Blo26]{Bloch26}
Andr\'e Bloch.
\newblock Sur les syst\`emes de fonctions uniformes satisfaisant \`a l'\'
  equation d'une vari\'et\'e alg\'ebrique dont l'irr\'egularit\'e d\'epasse la
  dimension.
\newblock {\em J. Math. Pures Appl.}, pages 19--66, 1926.

\bibitem[Bog77]{Bogomolov77}
F.~A. Bogomolov.
\newblock Families of curves on a surface of general type.
\newblock {\em Dokl. Akad. Nauk SSSR}, 236(5):1041--1044, 1977.

\bibitem[Bor69]{Borel_book69}
Armand Borel.
\newblock {\em Introduction aux groupes arithm\'etiques}.
\newblock Publications de l'Institut de Math\'ematique de l'Universit\'e de
  Strasbourg, XV. Actualit\'es Scientifiques et Industrielles, No. 1341.
  Hermann, Paris, 1969.

\bibitem[Bor72]{Borel72}
Armand Borel.
\newblock Some metric properties of arithmetic quotients of symmetric spaces
  and an extension theorem.
\newblock {\em J. Differential Geometry}, 6:543--560, 1972.
\newblock Collection of articles dedicated to S. S. Chern and D. C. Spencer on
  their sixtieth birthdays.

\bibitem[Bru14]{Brunebarbe}
Yohan Brunebarbe.
\newblock Symmetric differentials and variations of {H}odge structures.
\newblock {\em to appear in Crelle's journal}, 2014.

\bibitem[Bry79]{Brylinski79}
Jean-Luc Brylinski.
\newblock Propri\'et\'es de ramification \`a l'infini du groupe modulaire de
  {T}eichm\"uller.
\newblock {\em Ann. Sci. \'Ecole Norm. Sup. (4)}, 12(3):295--333, 1979.
\newblock With an appendix in English by Ken Baclawski.

\bibitem[Cla16]{Claudon-Bourbaki}
Beno\^it Claudon.
\newblock Semi-positivity of logarithmic cotangent bundle and
  {S}hafarevich-{V}iehweg's conjecture, after {C}ampana, {P}\u{a}un, {T}aji...
\newblock {\em arXiv: 1603.09568}, March 2016.

\bibitem[CP13]{Campana-Paun1}
Fr\'ed\'eric Campana and Mihai P\u{a}un.
\newblock Orbifold generic semi-positivity: an application to families of
  canonically polarized manifolds.
\newblock {\em arXiv: 1303.3169}, March 2013.

\bibitem[Del70]{Deligne70}
Pierre Deligne.
\newblock {\em \'{E}quations diff\'erentielles \`a points singuliers
  r\'eguliers}.
\newblock Lecture Notes in Mathematics, Vol. 163. Springer-Verlag, Berlin-New
  York, 1970.

\bibitem[Del71]{Deligne71}
Pierre Deligne.
\newblock Travaux de {S}himura.
\newblock In {\em S\'eminaire {B}ourbaki, 23\`eme ann\'ee (1970/71), {E}xp.
  {N}o. 389}, pages 123--165. Lecture Notes in Math., Vol. 244. Springer,
  Berlin, 1971.

\bibitem[Del79]{Deligne77}
Pierre Deligne.
\newblock Vari\'et\'es de {S}himura: interpr\'etation modulaire, et techniques
  de construction de mod\`eles canoniques.
\newblock In {\em Automorphic forms, representations and {$L$}-functions
  ({P}roc. {S}ympos. {P}ure {M}ath., {O}regon {S}tate {U}niv., {C}orvallis,
  {O}re., 1977), {P}art 2}, Proc. Sympos. Pure Math., XXXIII, pages 247--289.
  Amer. Math. Soc., Providence, R.I., 1979.

\bibitem[Del87]{Deligne87}
Pierre Deligne.
\newblock Un th\'eor\`eme de finitude pour la monodromie.
\newblock In {\em Discrete groups in geometry and analysis ({N}ew {H}aven,
  {C}onn., 1984)}, volume~67 of {\em Progr. Math.}, pages 1--19. Birkh\"auser
  Boston, Boston, MA, 1987.

\bibitem[DMR10]{Diverio-Merker-Rousseau10}
Simone Diverio, Jo{\"e}l Merker, and Erwan Rousseau.
\newblock Effective algebraic degeneracy.
\newblock {\em Invent. Math.}, 180(1):161--223, 2010.

\bibitem[FL13]{Friedman-Laza13}
Robert Friedman and Radu Laza.
\newblock Semialgebraic horizontal subvarieties of {C}alabi-{Y}au type.
\newblock {\em Duke Math. J.}, 162(12):2077--2148, 2013.

\bibitem[GG80]{Green-Griffiths79}
Mark Green and Phillip Griffiths.
\newblock Two applications of algebraic geometry to entire holomorphic
  mappings.
\newblock In {\em The {C}hern {S}ymposium 1979 ({P}roc. {I}nternat. {S}ympos.,
  {B}erkeley, {C}alif., 1979)}, pages 41--74. Springer, New York-Berlin, 1980.

\bibitem[Gro94]{Gross94}
Benedict~H. Gross.
\newblock A remark on tube domains.
\newblock {\em Math. Res. Lett.}, 1(1):1--9, 1994.

\bibitem[Hel78]{Helgason78}
Sigurdur Helgason.
\newblock {\em Differential geometry, {L}ie groups, and symmetric spaces},
  volume~80 of {\em Pure and Applied Mathematics}.
\newblock Academic Press, Inc. [Harcourt Brace Jovanovich, Publishers], New
  York-London, 1978.

\bibitem[HT06]{Hwang-To06}
Jun-Muk Hwang and Wing-Keung To.
\newblock Uniform boundedness of level structures on abelian varieties over
  complex function fields.
\newblock {\em Math. Ann.}, 335(2):363--377, 2006.

\bibitem[Iit82]{Iitaka_book}
Shigeru Iitaka.
\newblock {\em Algebraic geometry}, volume~76 of {\em Graduate Texts in
  Mathematics}.
\newblock Springer-Verlag, New York-Berlin, 1982.
\newblock An introduction to birational geometry of algebraic varieties,
  North-Holland Mathematical Library, 24.

\bibitem[Kaw80]{Kawamata80}
Yujiro Kawamata.
\newblock On {B}loch's conjecture.
\newblock {\em Invent. Math.}, 57(1):97--100, 1980.

\bibitem[Kli11]{Klingler11}
Bruno Klingler.
\newblock Local rigidity for complex hyperbolic lattices and {H}odge theory.
\newblock {\em Invent. Math.}, 184(3):455--498, 2011.

\bibitem[Lan86]{Lang86}
Serge Lang.
\newblock Hyperbolic and {D}iophantine analysis.
\newblock {\em Bull. Amer. Math. Soc. (N.S.)}, 14(2):159--205, 1986.

\bibitem[Laz04a]{Lazarsfeld04I}
Robert Lazarsfeld.
\newblock {\em Positivity in algebraic geometry. {I}}, volume~48 of {\em
  Ergebnisse der Mathematik und ihrer Grenzgebiete. 3. Folge. A Series of
  Modern Surveys in Mathematics [Results in Mathematics and Related Areas. 3rd
  Series. A Series of Modern Surveys in Mathematics]}.
\newblock Springer-Verlag, Berlin, 2004.
\newblock Classical setting: line bundles and linear series.

\bibitem[Laz04b]{Lazarsfeld04II}
Robert Lazarsfeld.
\newblock {\em Positivity in algebraic geometry. {II}}, volume~49 of {\em
  Ergebnisse der Mathematik und ihrer Grenzgebiete. 3. Folge. A Series of
  Modern Surveys in Mathematics [Results in Mathematics and Related Areas. 3rd
  Series. A Series of Modern Surveys in Mathematics]}.
\newblock Springer-Verlag, Berlin, 2004.
\newblock Positivity for vector bundles, and multiplier ideals.

\bibitem[McQ98]{McQuillan98}
Michael McQuillan.
\newblock Diophantine approximations and foliations.
\newblock {\em Inst. Hautes \'Etudes Sci. Publ. Math.}, (87):121--174, 1998.

\bibitem[Mil13]{Milne13}
J.~S. Milne.
\newblock Shimura varieties and moduli.
\newblock In {\em Handbook of moduli. {V}ol. {II}}, volume~25 of {\em Adv.
  Lect. Math. (ALM)}, pages 467--548. Int. Press, Somerville, MA, 2013.

\bibitem[Mok89]{Mok_book}
Ngaiming Mok.
\newblock {\em Metric rigidity theorems on {H}ermitian locally symmetric
  manifolds}, volume~6 of {\em Series in Pure Mathematics}.
\newblock World Scientific Publishing Co., Inc., Teaneck, NJ, 1989.

\bibitem[Mum77]{Mumford77}
David Mumford.
\newblock Hirzebruch's proportionality theorem in the noncompact case.
\newblock {\em Invent. Math.}, 42:239--272, 1977.

\bibitem[Nad89]{Nadel}
Alan~Michael Nadel.
\newblock The nonexistence of certain level structures on abelian varieties
  over complex function fields.
\newblock {\em Ann. of Math. (2)}, 129(1):161--178, 1989.

\bibitem[Nog91]{Noguchi91}
Junjiro Noguchi.
\newblock Moduli space of abelian varieties with level structure over function
  fields.
\newblock {\em Internat. J. Math.}, 2(2):183--194, 1991.

\bibitem[Och77]{Ochiai77}
Takushiro Ochiai.
\newblock On holomorphic curves in algebraic varieties with ample irregularity.
\newblock {\em Invent. Math.}, 43(1):83--96, 1977.

\bibitem[PW16]{Popa-Wu16}
Mihnea Popa and Lei Wu.
\newblock Weak positivity for {H}odge modules.
\newblock {\em arXiv: 1511.00290}, March 2016.

\bibitem[Rou13]{Rousseau13}
Erwan Rousseau.
\newblock Hyperbolicity, automorphic forms and {S}iegel modular varieties.
\newblock {\em arXiv:1302.4723, To appear in Annales Scientifiques de l'ENS},
  February 2013.

\bibitem[Sak79]{Sakai}
Fumio Sakai.
\newblock Symmetric powers of the cotangent bundle and classification of
  algebraic varieties.
\newblock In {\em Algebraic geometry ({P}roc. {S}ummer {M}eeting, {U}niv.
  {C}openhagen, {C}openhagen, 1978)}, volume 732 of {\em Lecture Notes in
  Math.}, pages 545--563. Springer, Berlin, 1979.

\bibitem[Sch73]{Schmid73}
Wilfried Schmid.
\newblock Variation of {H}odge structure: the singularities of the period
  mapping.
\newblock {\em Invent. Math.}, 22:211--319, 1973.

\bibitem[Sim92]{Simpson92}
Carlos~T. Simpson.
\newblock Higgs bundles and local systems.
\newblock {\em Inst. Hautes \'Etudes Sci. Publ. Math.}, (75):5--95, 1992.

\bibitem[SZ10]{Sheng-Zuo10}
Mao Sheng and Kang Zuo.
\newblock Polarized variation of {H}odge structures of {C}alabi-{Y}au type and
  characteristic subvarieties over bounded symmetric domains.
\newblock {\em Math. Ann.}, 348(1):211--236, 2010.

\bibitem[UY10]{Ullmo-Yafaev10}
Emmanuel Ullmo and Andrei Yafaev.
\newblock Points rationnels des vari\'et\'es de {S}himura: un principe du
  ``tout ou rien''.
\newblock {\em Math. Ann.}, 348(3):689--705, 2010.

\bibitem[Vie83]{Viehweg83}
Eckart Viehweg.
\newblock Weak positivity and the additivity of the {K}odaira dimension for
  certain fibre spaces.
\newblock In {\em Algebraic varieties and analytic varieties ({T}okyo, 1981)},
  volume~1 of {\em Adv. Stud. Pure Math.}, pages 329--353. North-Holland,
  Amsterdam, 1983.

\bibitem[Vie95]{Viehweg_book}
Eckart Viehweg.
\newblock {\em Quasi-projective moduli for polarized manifolds}, volume~30 of
  {\em Ergebnisse der Mathematik und ihrer Grenzgebiete (3) [Results in
  Mathematics and Related Areas (3)]}.
\newblock Springer-Verlag, Berlin, 1995.

\bibitem[VZ02]{Viehweg-Zuo02}
Eckart Viehweg and Kang Zuo.
\newblock Base spaces of non-isotrivial families of smooth minimal models.
\newblock In {\em Complex geometry ({G}\"ottingen, 2000)}, pages 279--328.
  Springer, Berlin, 2002.

\bibitem[Wei86]{Weissauer86}
Rainer Weissauer.
\newblock Untervariet\"aten der {S}iegelschen {M}odulmannigfaltigkeiten von
  allgemeinem {T}yp.
\newblock {\em Math. Ann.}, 275(2):207--220, 1986.

\bibitem[Zuc81]{Zucker81}
Steven Zucker.
\newblock Locally homogeneous variations of {H}odge structure.
\newblock {\em Enseign. Math. (2)}, 27(3-4):243--276 (1982), 1981.

\bibitem[Zuo00]{Zuoneg}
Kang Zuo.
\newblock On the negativity of kernels of {K}odaira-{S}pencer maps on {H}odge
  bundles and applications.
\newblock {\em Asian J. Math.}, 4(1):279--301, 2000.
\newblock Kodaira's issue.

\end{thebibliography}

\end{document}